\documentclass{article}
\usepackage{amsmath, amsthm, amssymb, setspace, graphics,  graphicx}
\usepackage[all]{xy}
\usepackage{url}
\newtheorem{thm}{Theorem}[subsection] 
\newtheorem{pro}[thm]{Proposition} 
\newtheorem{lem}[thm]{Lemma} 
\newtheorem{cor}[thm]{Corollary} 
\newtheorem{conj}[thm]{Conjecture}
\newtheorem{prob}[thm]{Problem}
\theoremstyle{definition} 
\newtheorem{defn}[thm]{Definition} 
\theoremstyle{remark} 
\newtheorem{rem}[thm]{Remark}
\newtheorem{exa}[thm]{Example}

\newcommand{\CC}{\mathbb C}

\newcommand{\NN}{\mathbb N}
\newcommand{\PP}{\mathbb P}
\newcommand{\QQ}{\mathbb Q}

\newcommand{\UU}{\mathbb U}

\newcommand{\WW}{\mathbb W}
\newcommand{\ZZ}{\mathbb Z}

\newcommand{\Ocal}{\mathcal O}
\newcommand{\Acal}{\mathcal A}
\newcommand{\Ccal}{\mathcal C}
\newcommand{\Lcal}{\mathcal L}
\newcommand{\Hcal}{\mathcal H}
\newcommand{\Fcal}{\mathcal F}

\newcommand{\Vcal}{\mathcal V}

\newcommand{\Mcal}{\mathcal M}
\newcommand{\Ecal}{\mathcal E}
\newcommand{\Bcal}{\mathcal B}
\newcommand{\Xcal}{\mathcal X}
\newcommand{\Qcal}{\mathcal Q}
\newcommand{\Dcal}{{\mathcal D}}
\newcommand{\iii}{{h}}

\newcommand{\id}{\operatorname{id}}
\newcommand{\Aut}{\operatorname{Aut}}
\newcommand{\Jac}{\operatorname{Jac}}

\newcommand{\mult}{\operatorname{mult}}

\newcommand{\Hom}{\operatorname{Hom}}

\newcommand{\Diag}{\operatorname{Diag}}

\newcommand{\Res}{\operatorname{Res}}
\newcommand{\al}{\alpha}
\newcommand{\be}{\beta}

\newcommand{\fie}{\varphi}

\newcommand{\la}{\lambda}

\newcommand{\ev}{{\rm ev}}

\newcommand{\Jgr}{\langle j_W\rangle}
\newcommand{\Jnw}{\langle j\rangle}
\newcommand{\ctop}{{c}_{\mathrm{top}}}

\newcommand{\cxi}{\mathtt{i}}
\newcommand{\rk}{\operatorname{rk}}

\newcommand{\age}{\operatorname{age}}
\newcommand{\hess}{\operatorname{hess}}

\newcommand{\ch}{{\rm ch}}
\newcommand{\td}{\operatorname{td}}
\newcommand{\Sym}{{\operatorname{Sym}}}
\newcommand{\textsum}{{\textstyle{\sum}}}

\providecommand{\abs}[1]{\lvert#1\rvert}

\def\vir{{\rm vir}}

\def\pmmu{{\pmb \mu}}

\def\W{\mathbb{W}}
\def\RW{\operatorname{FJRW}}
\def\CR{\operatorname{CR}}
\def\GW{\operatorname{GW}}

\def\ol{\overline}

\def\wt{\widetilde}
\def\wh{\widehat}

\newcommand{\MMM}{\overline{\mathcal M}}
\newcommand{\CCC}{\mathcal C}

\begin{document}

\title{
\textbf{A global mirror symmetry framework for the Landau--Ginzburg/Calabi--Yau correspondence}}

\author{Alessandro Chiodo and Yongbin Ruan}
\maketitle

\begin{abstract}
{We show how the Landau--Ginzburg/Calabi--Yau correspondence
for the quintic three-fold can be cast into a global
mirror symmetry framework.
Then we draw inspiration from Berglund--H\"ubsch  
mirror duality construction to
provide an analogue picture
featuring all Calabi--Yau hypersurfaces within weighted projective spaces
and certain quotients by finite abelian group
actions.}
 \end{abstract}

\vspace*{6pt}\tableofcontents

\section{Introduction}
We survey FJRW theory introduced by Fan, Jarvis, and the second author 
for the Landau--Ginzburg model following ideas of 
Witten.
We review its connection 
to related work and we provide a prospectus on the 
ideas guiding the long term development 
of FRJW theory. The paper also contains 
some new results on foundational aspects of FRJW theory. 

The theory can be motivated as a tool for the computation of Gromov--Witten invariants.
Almost twenty years ago  
a correspondence was proposed 
(see \cite{VW89} and \cite{Wi93b}) in order to 
connect two areas of physics:
the Landau--Ginzburg (LG) model and Calabi--Yau (CY)
geometry.
In simple terms,
the geometry of certain CY spaces
is expected to be completely encoded by another
geometrical object, the LG model,
which is in many cases easier to study.
The case of the quintic three-fold 
illustrates this well: a smooth hypersurface defined in $\PP^4$ by a 
homogeneous degree-five polynomial plays a central role in Gromov--Witten theory since 
its early developments. Whereas in genus zero the theory has been completely elucidated in 
 \cite{Givental} and \cite{LLY} matching the mirror symmetry conjecture,
for positive genus the theory is largely unknown: it has been
    determined by Zinger \cite{Zi} for $g=1$ and  is still wide open for $g>1$ despite the
joint effort of mathematicians and physicists over the last twenty years.  
From the point of view of theoretical physics, the most advanced effort is
Huang, Klemm, and Quackenbush's speculation \cite{HKQ} via a physical argument; it is striking however
that, even with these far-reaching techniques,
there is no prediction beyond $g=52$.
    A natural idea to approach the higher genus cases
    consists in providing a mathematical statement of the physical LG-CY
    correspondence and 
    using the computational power of the LG singularity model to determine the 
    higher genus
    Gromov--Witten invariants of the CY manifold.
    From the mathematical point of view, this conceptual framework is largely incomplete: whereas Gromov--Witten (GW)
    theory embodies all the relevant information on the CY side,
    it is not clear which theory plays the
    same role on the LG side. This is likely to be interesting in its own right; 
    for instance, in a different context, the LG-CY correspondence led to Orlov's equivalence 
    between the derived category
    of complexes of coherent
    sheaves and matrix factorizations (see \cite{Orlov}, \cite{H} and \cite{Ko1}).

    In \cite{FJR1, FJR2, FJR3}, Fan, Jarvis and the second author construct 
    such a candidate quantum theory of singularities: FJRW theory.
    In intuitive terms,  
    GW theory
    may be regarded as the study of the solutions of the Cauchy--Riemann
    equation $\ol{\partial}f=0$ for
    the map $f\colon C\rightarrow
    X_W$, where $C$ is a compact Riemann surface and $X_W$ is a degree-$N$ hypersurface within a projective space with $N$ homogeneous coordinates. 
    On the other hand,
    in the LG singularity model, we treat $W$ as a 
    holomorphic function on $\CC^N$. From this perspective, FJRW theory is about solving a generalized 
    PDE attached to $W$ rather than classifying holomorphic maps from a compact Riemann
    surface $\Sigma$ to $\CC^N$. The idea comes from Witten's conjecture 
    \cite{Wi1} stated in the early 90's and soon proven by Kontsevich \cite{Ko}: 
    the intersection theory
    of Deligne and Mumford's moduli of curves
    is governed by the KdV integrable
    hierarchy---\emph{i.e.}~the integrable system corresponding to the $A_1$-singularity.
    Witten generalized Deligne and Mumford's spaces to new moduli spaces governed by
    integrable hierarchies attached to more general singularities.
    To this effect, he considers 
    the PDE
    \begin{equation}\label{eq:pde}
\ol{\partial}s_j+\overline{{\partial_jW}(s_1, \cdots, s_N)}=0,
     \end{equation}
    where $W$ is the same polynomial defining $X_W$
    and $\partial_jW$ is the derivative with respect to the $j$th variable. 
    Faber, Shadrin and Zvonkine proved this conjecture for $A_n$-singularities.
    Fan, Jarvis, and the second author \cite{FJR1, FJR2, FJR3} extended Witten approach to any singularity and 
    genealized the proof of Witten's statement to all simple singularities.
    In this way FJRW theory
    plays the role of Gromov--Witten theory
    on the LG side for any isolated singularity defined by a quasihomogeneous
    polynomial. The Witten equation
    should be viewed as the counterpart to the
    Cauchy--Riemann equation: when we pass to the LG singularity model 
    we replace the linear Cauchy--Riemann equation on a
    nonlinear target with the nonlinear Witten equation on a linear target.

    Three years ago, a program was launched by the authors 
    in order to establish the LG-CY correspondence mathematically.
    Since then, a great deal of progress has been
    made: the proof of classical mirror symmetry statements via the LG model
    (by the authors \cite{CR_AIM} and Krawitz \cite{Kr}),
    the modularity of the Gromov--Witten theory of elliptic orbifold $\PP^1$
    (see Krawitz--Shen \cite{KSh} and work by the second author in collaboration with Milanov \cite{MR}) and the connection to
    Orlov's equivalence (by the authors in collaboration with Iritani \cite{CIR}). In this survey article,
    we report on some of the progress within a common framework and we complement at several points our treatment 
    of the quintic threefold  \cite{ChRu}.

\subsection{LG-CY correspondence and ``global'' mirror symmetry}
So far we have presented the LG-CY correspondence from the point of view of 
the open problem of computing GW theory. 
The framework of mirror symmetry, however,  allows us to
   recast this transition from CY geometry to the LG side within a
   geometric setup involving a wider circle of ideas. This is the main focus of this paper.

Recall that mirror symmetry asserts a duality among CY three-folds
   exchanging the $A$ model invariants with the  $B$ model invariants. Naively,
   the $A$ model contains information such
   as the K\"ahler structure and Gromov--Witten invariants,
   while the  $B$ model contains information
   such as the complex structure and period integrals.
   From a global point of view,
   this picture cannot be entirely satisfactory,
   because the complex moduli space has a nontrivial
   topology while the K\"ahler moduli space does not.

\subsubsection{Cohomological mirror symmetry}
   Let us illustrate this issue by means of
   the example which inspired the whole phenomenon of mirror symmetry \cite{CDGP}.
   On the one side of the mirror we have the quintic three-fold
   \begin{equation}\label{eq:quinticFermatCY}
   X_W=\{x_1^5+x_2^5+x_3^5+x_4^5+x_5^5=0\}{\subset \PP^4}
   \end{equation}
   equipped with a natural holomorphic three-form
   $\omega={dx_1\wedge dx_2\wedge dx_3 }/{x_4^4}$
   (written here in coordinates with $x_5=1$).
   On the other side we take the quotient of $X_W$ by
   the group $G\cong (\ZZ_5)^4$
   spanned by
   $x_i\mapsto \al x_i$ with $\al^5=1$ for all $i=1, \dots, 5$
   subject to the condition that $\omega$ is
   preserved\footnote{In other words, each diagonal transformation
   $\Diag(\al_1\in \pmmu_5,\dots,\al_5\in \pmmu_5)$
   should satisfy $\det=\prod_i \al_i=1$.}.
   The quotient scheme $X_W/G$ is singular;
   but there is a natural, canonically defined, resolution $Y=(X_W/G)^{\rm res}$ which is
   again a CY variety.

   In general the existence of resolutions of CY type is not
   guaranteed. But we can rephrase things in higher generality
   in terms of orbifolds:
   let us mod out $G$ by the kernel of $G\to \Aut(X_W)$, the group
   spanned by the diagonal symmetry $j_W$  scaling all coordinates
   by the same primitive fifth root $\xi_5$.  Then,
   the quotient of $X_W$ by $\wt G=G/\langle j_W\rangle$ equals $X_W/G$
   and the group $\wt G$ acts faithfully.
   In this way,
   the resolution $Y$ may be
   equivalently replaced by the smooth quotient stack (orbifold)
   \begin{equation}\label{eq:mirror}
   X_W^{{\vee}}=[X_W/\wt G]
   \end{equation}
   (a cohomological equivalence between $Y $ and $X_W^{\vee}$ holds under the condition
   that the stabilizers are
   nontrivial only in codimension $2$).

   The odd cohomology (primitive cohomology)
   of $X_W^{\vee}$ is four-dimensional and
   unusually simple: the odd-degree Hodge numbers equal $(1,1,1,1)$ and mirror the
   four hyperplane sections $\pmb 1, H, H^2, H^3$
   of the projective hypersurface $X_W$
   \begin{equation*}\begin{matrix}
h^{p,q}(X_W^{{\vee}})=\quad & &   &1&   &   \\
 & &0  & &0  &   \\
 &0&   &101&   &0  \\
\hspace{2cm} 1& &  1&   &1  & &1\\
 &0&   &101&   &0  \\
 & &0  & &0  &   \\
\quad  & &   &1&   &&&
\end{matrix}\qquad \qquad\qquad\qquad\begin{matrix}\label{eq:quinticdiamond}
h^{p,q}(X_W)=& & &   &1&   &   \\
& & &0  & &0  &   \\
& &0&   &1&   &0  \\
&1& &101& &101& &1.\\
& &0&   &1&   &0  \\
& & &0  & &0  &   \\
\quad & & &   &1&   &&&\end{matrix}
\end{equation*}
Indeed, this is part of the
   cohomological mirror symmetry
   \begin{equation}\label{eq:classMS}
   h^{p,q}(X_W)=h^{\dim -p,q}(X_W^{{\vee}}).
   \end{equation}

\subsubsection{Mirror symmetry at the large complex structure point}
   We further illustrate mirror symmetry for this example with special attention to
   the difference in global geometry between the two sides.
   On one side of the mirror, for $X_W$, we consider the
   (complexified) K\"ahler moduli space --- a \emph{contractible} one-dimensional complex
   space which should be
   regarded as an $A$ side invariant
   $$\Acal(X_W).$$
   On the other side of the mirror we consider a $B$ model invariant: the
   (complex structure) deformations of $[X_W/\wt G]$. These are
   actually deformations of $X_W$ preserved by the action
   of $\wt G=G/\langle j_W\rangle$. We get the Dwork family
   $$X_{W,\psi}=\left\{x_1^5+x_2^5+x_3^5+x_4^5+x_5^5+5\psi\prod_{i=1}^5 x_i=0\right\},$$
   on which $\wt G$ operates by preserving the fibres and the form
   $\omega_{\psi}={dx_1\wedge dx_2\wedge dx_3 }/({x_4^4-\psi x_1x_2x_3})$
   yielding a family of CY
   orbifolds $X^{\vee}_{W,\psi}$ over an open subset of $\PP^1_\psi$
   (the complement of the divisor
   where singularities occur).
   In fact,  for $\al^5=1$,  we can
   let the  diagonal symmetry $x_i\mapsto \al x_i$
   operate on the family so that the action identifies the fibre $X^{\vee}_{W,\psi}$
   over
   $\psi$ with the isomorphic fibre $X^{\vee}_{W,\al\psi}$
   over $\al \psi$. Therefore,
   the Dwork family is ultimately a family of three-dimensional CY orbifolds over
   $[\PP^1/\ZZ_5]$. Write $t=\psi^5$;
   then the new family is regular off
   $t=\infty$ and $t=1$.
   These limit points alongside with the stack-theoretic point $t=0$
   are usually referred to as \emph{special limit points}; more precisely,
   $0,\infty,$ and $1$
   are referred to as the \emph{Gepner point}, the \emph{large complex structure point},
   and the
   \emph{conifold point}. Unlike the K\"ahler moduli space, this moduli space of
   complex structures is \emph{not contractible}. For this reason, mirror
   symmetry has been studied as an identification
   between the above contractible K\"ahler moduli space $\Acal(X_W)$ and a
   contractible neighborhood of the large complex structure
   point $t=\infty$
   $$\Bcal(X^{{\vee}}_{W,\infty}).$$ This leads to a formulation
   of mirror symmetry as a local statement matching the $A$ model to the  $B$ model
   restricted to a neighborhood of the large complex structure point.
   Consider the bundle over $\Bcal(X_{W,\infty}^{{\vee}})$ minus the origin with
   four-dimensional
   fibre $H^3 (X_{W,t}^\vee , \CC)$
   over $t \in \Bcal(X_{W,\infty}^{{\vee}})$.
   There is, of course, a flat connection, the Gauss--Manin connection,
   given by the local system $H^3(X_{W,t}^{\vee} , \ZZ) \subset H^3(X_{W,t}^{\vee} , \CC)$.
   Dubrovin has shown how to use Gromov--Witten invariants
   to put a flat connection
   on the four-dimensional bundle with
   fibre $H^{\ev}(X_{W})$ over $\Acal({X_W})$.
   Under a suitable identification (mirror map)
\begin{equation}\label{eq:LCLP_MS}
 \xymatrix@C=1.1cm{
\Bcal(X_{W,\infty}^{{\vee}})
\ar[d]^{\cong}
\\
\Acal(X_W)\ar[u]
}
\end{equation}
   the two structures are identified (Givental \cite{Gi}, Lian--Liu--Yau \cite{LLY}).
   This local point
   of view dominated the mathematical study of mirror symmetry for the last twenty years.

\subsubsection{Global mirror symmetry}
   It is natural to extend our study to
   the entire moduli space $[\PP^1/\ZZ_5]$
   and to all the special limits.
   Such a {\em global} point of view underlies a large part of the
   physics literature on the subject and
   leads naturally
   to the famous holomorphic anomaly equation \cite{BCOV}
   and, in turn, to the above
   mentioned spectacular physical predictions \cite{HKQ} on
   Gromov--Witten invariants of  the quintic three-fold up to genus $52$.
   In the early 90's, a physical solution was proposed to
   complete the K\"ahler moduli space by
   including other \emph{phases} \cite{MO, Wi93b}.
   As we shall illustrate, for the quintic three-fold,
   two phases arise in the
   $A$ model: the CY geometry and the LG phase.
   Whereas the CY geometry of the quintic has already
   been identified by mirror symmetry to a neighborhood
   of the large complex structure limit
   point $\Bcal(X_{W,\infty}^{\vee})$,
   the LG phase is expected to be mirror
   to the neighborhood of the Gepner point at $0$
   $$\Bcal(X_{W,0}^{{\vee}}).$$
   Then, the LG-CY correspondence can be interpreted as
   an analytic continuation from
   the Gepner point to the large complex structure point.
   From this point of view, the LG-CY correspondence should
   be viewed as a step towards global mirror symmetry.

   From a purely mathematical point of view
   it may appear difficult to make sense of
   such a transition of the CY quintic three-fold
   into a different ``phase''. Fortunately, Witten has illustrated this
   in precise mathematical terms as a variation of stability conditions in geometric invariant theory, \cite[\S4]{Wi93b}.
   Let us consider the explicit example of the Fermat quintic Calabi--Yau three-fold: let $Y=\CC^{6}$ with coordinates $x_1,\dots,x_5$  and  
   $p$  and let   $\CC^*$ act  as  $$x_i\mapsto \la x_i, \ \forall i;\qquad \qquad p\mapsto \la^{-5}p.$$
   The presence of nonclosed orbits prevents us from defining a geometric quotient.
 In order to obtain a geometric quotient, one should necessarily restrict to 
open $\CC^*$-invariant subsets $\Omega$ of $V\cong \CC^6$
   for which $\Omega/\CC^{*}$ exists.
   The geometric invariant theory (GIT) yields two maximal possibilities: 
   the sets $\Omega_1= \{\pmb x\neq 0\} $  yielding $\Ocal(-5)$ as a quotient by $\CC^*$ and
   the set $\Omega_2=\{p\neq 0\}$ yielding the
   orbifold $[\CC^5/\ZZ_5]$.
   If one adds to the picture
   a $\CC^*$-invariant holomorphic function such as
   $\wt W(p,x_1,\dots,x_5)=pW(x_1,\dots,x_5)=p\sum_i x_i^5$
   the two geometric models ultimately reduce
   to the Fermat quintic $X_W$ and to the singularity at the origin of
   \begin{equation}\label{eq:LGphase}
W=\sum_i x_i^5\colon [\CC^5/\ZZ_5]\longrightarrow \CC.
   \end{equation}

\medskip

   On $\Bcal(X_{W,0}^{{\vee}})$ consider the bundle with
   fibre $H^3 (X_{W,t}^{\vee}, \CC)$ over the point $t$.
   There is again the flat Gauss--Manin connection induced
   by the local system $H^3(X_{W,t}^{\vee} , \ZZ)
   \subset H^3(X_{W,t}^{\vee} , \CC)$.
   The work of Fan, Jarvis, and the second author \cite{FJR1} yields ---
   via Dubrovin connection ---
   a flat connection on a vector bundle on a contractible one-dimensional
   space $$\Acal(W, \ZZ_5)$$
   attached to \eqref{eq:LGphase}. (For the abstract formalism
   of Dubrovin connection we refer to Iritani \cite{Ir}.)
   The fibre of this bundle is a four-dimensional
   state space attached to the singularity $W\colon [\CC^5/\ZZ_5]\to \CC$
   (see \S\ref{subsect:states}).
   Under a suitable identification (mirror map)
\begin{equation}\label{eq:GepnerMS}
 \xymatrix@C=1.1cm{
\Bcal(X_{W,0}^{{\vee}})
\ar[d]^{\cong}
\\
\Acal(W,\ZZ_5)\ar[u]
}
\end{equation}
   we match the two structures in \cite{ChRu}.
The LG-CY correspondence can be now carried out via \eqref{eq:LCLP_MS} and
   \eqref{eq:GepnerMS}
   on the $B$ side via the local system induced by
   the family of CY orbifolds $X_{W,t}^{\vee}$ with $t$ varying in
   $(\PP^1)^\times=\PP_t^1\setminus\{0,1,\infty\}.$ Consider Figure \ref{fig:picture},
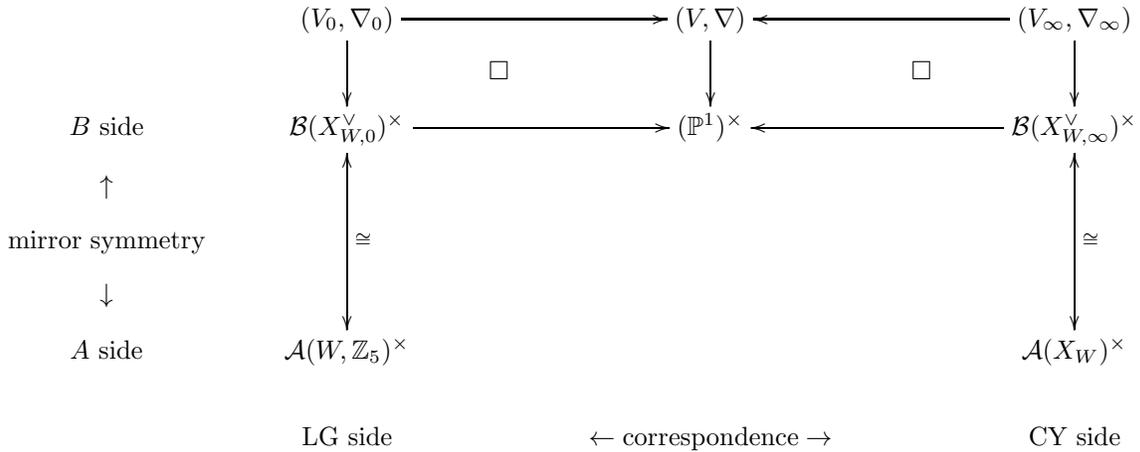
\begin{figure}[h]
\begin{equation*}
  \xymatrix@R=.2cm{
&(V_0,\nabla_0)\ar[dd] \ar[rr] && (V,\nabla)\ar[dd]&&\ar[ll] (V_\infty,\nabla_\infty)\ar[dd] \\
&&\square& &\square&\\
\text{$B$ side}&\Bcal(X_{W,0}^{{\vee}})^\times\ar[rr]
\ar[dddd]^{\cong} &&(\PP^1)^\times&&
\Bcal(X_{W,\infty}^{{\vee}})^\times\ar[ll]
\ar[dddd]^{\cong}
\\
\uparrow&\\
{\text{mirror symmetry}}&\\
\downarrow&\\
\text{$A$ side}&\Acal(W,\ZZ_5)^\times\ar[uuuu] &&&& \Acal(X_W)^\times\ar[uuuu]\\\\
& \text{LG side} &&\leftarrow\text{correspondence}\rightarrow && \text{CY side}
}
\end{equation*}
\caption{Casting LG-CY correspondence within the global mirror symmetry framework.}
\label{fig:picture}
\end{figure}
where the notation $(\ \ )^\times$ stands for $(\ \ )\setminus\{\text{special points $0$, $1$ and $\infty$}\}$,
the horizontal maps
to $(\PP^1)^\times$ are the natural inclusions, and $V$,
$V_0$ and $V_\infty$ are the four-dimensional bundles
with fibre $H^3(X_{W,t}^{\vee},\CC)$ equipped with the respective
Gauss--Manin connections $\nabla$.

\newpage

On the CY side, \emph{i.e.}~on $\Acal(X_W)^\times$, the isomorphism \eqref{eq:LCLP_MS}
and the study of the variation of the Hodge structure of $X_{W,t}^{\vee}$ on
$\Bcal(X_{W,\infty})^\times$
allow us to associate to a given basis of $H^{\ev}(X_W)$ a basis of
multivalued functions from $\Acal(X_W)^\times$ to $H^{\ev}(X_W)$ which are flat with respect to
Dubrovin connection. This  amounts to solving Gromov--Witten theory for
$X_W$ in genus zero.
The analogous problem holds on
$\Acal(W,\ZZ_5)^\times$ on the LG side; it is solved
via \eqref{eq:GepnerMS} by computing FJRW theory for $(W,\ZZ_5)$ in
genus zero.
Furthermore, via analytic continuation, we can extend
the bases of flat sections globally on $(\PP^1)^\times$
and find a change of bases matrix
   \begin{equation}\label{eq:LGCYsympl}
   \mathbb U_{\text{\rm LG-CY}}.\end{equation}
This is explicitly computed in \cite{ChRu} for the Fermat quintic and
is independent of the base parameter $t$ on $(\PP^1)^\times$.
In \cite{CIR} we provide a geometric interpretation in terms of
Orlov's equivalence relating the bounded derived category of coherent sheaves on 
$X_W$ and the category of matrix factorizations attached to $(W,\ZZ_5)$ (see \S\ref{subsect:CIR}).
This point of view is interesting
in its own right, because it short-circuits
mirror symmetry in the diagram
and provides a conceptual explanation to
   the fact that $\mathbb U_{\text{\rm LG-CY}}$ is symplectic (a crucial fact for
   Conjecture \ref{conj:LGCY}).  We refer to \S\ref{subsect:CIR}.

This paper focuses on the generalization of Figure \ref{fig:picture}.
Indeed, the richness of the mirror symmetry picture calls for generalizations in
several directions and makes the whole process of figuring out all sides of
the story very entertaining: some
corners of Figure \ref{fig:picture}  can be more rapidly understood and shed light on
entire the picture.
For instance, there are cohomological
aspects that can be treated
in very high generality; \emph{e.g.}~we provide results
applying to finite group quotients of
CY hypersurfaces in weighted projective spaces (see Section \ref{sect:classMS}).
On the other hand,
the \emph{quantum} cohomological aspects are intrinsically more difficult; there,
the results are almost exclusively
limited to certain well behaved ambient spaces (Gorenstein weighted
projective spaces); in this way, the purely cohomological results can be used to formulate conjectures.

   \subsection{Structure of the paper}
   In Section \ref{sect:FJRW}, we will review Fan--Jarvis--Ruan--Witten (FJRW) theory
   which plays a crucial role in these recent developments.
   In Section \ref{sect:classMS} we will generalize point \eqref{eq:classMS}
   above discussing
   results of Berglund, H\"ubsch and Krawitz on LG phases and of the authors on LG-CY
   cohomological correspondence.
   In Section \ref{sect:genuszeroLGCY} we focus on the quantum
   counterpart of these theorems; \emph{i.e.}~we provide correspondences
   involving the enumerative geometry of curves.
   This section is structured in four parts.
   We will state the LG-CY conjecture and review recent results
   \S\ref{subsect:LGCY}.
   Then, we will cast it in a global mirror
   symmetry framework in \S\ref{subsect:globalMS}.
   We will present in \S\ref{subsect:CYP1}
   a result going very far in providing evidence for this
   global mirror symmetry framework in higher genus \cite{KSh} \cite{MR}.
   Finally, in \S\ref{subsect:CIR} we will provide
   an independent
   interpretation
   of the LG-CY correspondence via Orlov's equivalence (work in collaboration with Iritani).

\subsection{Acknowledgements}
     We wish to thank Tom Coates, Huijun Fan, Tyler Jarvis, and Dimitri 
     Zvonkine for the collaborations
     leading to the current project. Special
     thanks go to Mina Aganagic, Kentaro Hori,
     Albrecht Klemm, and Edward Witten for many
     insightful conversations which greatly influenced this paper.
     Finally we would like to express our gratitude to the
     Institut Fourier, Universit\'e de Grenoble, for
     making this work possible by providing the second author
     with hospitality and excellent working conditions
     and by giving us the opportunity to contribute to
     the present special volume of the Annales de l'Institut Fourier.
     This paper was finalized during the first author's 
     visit at
     the University of Kyoto funded by a JSPS fellowship; 
     a special thank goes to these institutions and in particular to 
     Hiroshi Iritani for 
     the collaboration on this project. 
     We are grateful to the anonymous referee for his time 
     and effort in proving comments which greatly improved this paper.

\section{Fan--Jarvis--Ruan--Witten theory} \label{sect:FJRW}
In this section, we review Fan, Jarvis, and Ruan's construction
    of  ``quantum'' singularity theory based on Witten's
    partial differential equation \eqref{eq:pde}.
The treatment given here is more general than that appearing in \cite{ChRu}  and complements 
\cite{FJR1}. 
    The theory provides us with the LG side of the correspondence.
    The input for the theory is a pair $(W, G)$ where $W$ is a ``nondegenerate''
    quasihomogeneous polynomial $W\colon \CC^N\rightarrow \CC$
 and $G$ is a group of diagonal symmetries of $W$.

We say that $W\colon \CC^N\to \CC$
is quasihomogeneous (or weighted homogeneous) polynomial
of type $(q_1,\dots,q_N)$ for $q_j \in \QQ_{>0}$ if the following condition is satisfied.
Let $W=\sum_{i=1}^s \gamma_i \prod x_j^{m_{i,j}}$
with $m_{i,j}\in \ZZ$, $m_{i,j}\ge 0$ and $\gamma_i\neq 0$; then $\sum_{j=1}^N m_{i,j}q_j=1$.
Equivalently, with a slight abuse of notation, we write
$$W(\lambda^{q_1}x_1, \dots, \lambda^{q_N}x_N)=\lambda W(x_1,
\dots, x_N)$$
and we refer to $q_1,\dots,q_N$ as the \emph{charges} of $W$.
The polynomial $W$ is \emph{nondegenerate} if: (1) $W$
defines a unique singularity at zero; (2) the choice of $q_1,\dots, q_N$
is unique.

An element $g\in GL(\CC^N)$ is a \emph{diagonal symmetry
    of $W$} if $g$ is a diagonal matrix of the
    form $\Diag(\lambda_1, \dots, \lambda_N)$ such that
    $$W(\lambda_1 x_1, \dots, \lambda_N x_N)=W(x_1, \dots,
    x_N).$$
    We will use $\Aut(W)$ to denote the group of all diagonal symmetries and
    we will refer to it as the maximal group of diagonal symmetries.
    It is easy to see that this group is finite (see for
    instance \cite{FJR1}).
    The group is also nontrivial since it contains the element
    $j_W=\Diag(e^{2\pi \cxi q_1}, \dots, e^{2\pi \cxi q_N})$.

FJRW theory applies to a pair
    $(W, G)$, where $G\subseteq \Aut(W)$. Two conditions will
    naturally arise in the rest of the paper; their role is specular in the sense
    of mirror symmetry.
    We will say that $G\subseteq \Aut(W)$ is \emph{$A$-admissible} if
    $j_W$ is contained in $G$. We will say that it is  \emph{$B$-admissible}
    if $G\subseteq SL(\CC^N)$; \emph{i.e.}~if $G$ is included in
    $SL_W=SL(\CC^N)\cap \Aut(W)$.

\subsection{$A$ model state space}\label{subsect:states}
The state space was introduced by Fan, Jarvis, and Ruan
\cite{FJR1, FJR2, FJR3} as part of the moduli theory of the Witten
equation. We provide a
purely mathematical definition, independent from the mirror symmetry motivation.

\subsubsection{Lefschetz thimbles from the classical point of view}
Consider $W\colon \CC^N\rightarrow \CC$. Let us recall
some important facts on the
relative homology of $(\CC^N, W^{-1}(S_M^+))$
where $S_M^+$ is the half-plane $\{z\in \CC\mid {\rm Re}z> M\}$ for $M>0$.
We denote it by $$H_N(\CC^N,W^{+\infty};\ZZ)$$
with $W^{+\infty}=W^{-1}(S_M^+)$ and,
by abusing notation, we refer  to it as the {\em space of Lefschetz thimbles}.
\begin{rem}
Due to the nondegeneracy condition, the origin is the only critical point of $W$, and
$W$ is a fibre bundle on $\CC^{\times}$; for $N>1$, since $\CC^N$ is contractible,
we can regard the above relative cohomology as the
homology (with compact support) of rank $N-1$ of
the fibre over a point
of $S^+_M$ (\cite[1.1]{Pham}).
\end{rem}
\begin{rem}\label{rem:Gspace}
Standard arguments (\cite[ch. 2]{AGV} and \cite[(5.11)]{Loo}) show that
the space of Lefschetz thimbles is freely generated by as many
generators $\mu$ as the complex dimension of the local algebra $\Qcal_W$.
Furthermore,
due to \cite{OS} and \cite{Wa1}, for any $G\subseteq \Aut(W)$,
its dual $\Hom(H,\CC)$ is isomorphic
\emph{as a $G$-space} to
$dx_1\wedge \dots\wedge dx_N\cdot \Qcal_W$
(where a diagonal symmetry $g=\Diag(\la_1,\dots,\la_N)$ acts
on  $dx_1\wedge \dots\wedge dx_N$ by
multiplication by $\prod_j \la_j$).
\end{rem}
\begin{rem}
A nondegenerate pairing can be defined in the
following sense. Following \cite[\S8, Step 2]{Hertling} and \cite{Pham},
we consider the relative homology
$$H_N(\CC^N,W^{-\infty};\ZZ),$$
where $W^{-\infty}$ denotes $W^{-1}(S^-_M)$ and
$S^-_M$ is the half-plane $\{z\in \CC\mid {\rm Re}z<-M\}$ for $M>0$.
The intersection form for Lefschetz thimbles with boundaries in
$W^{+\infty}$ and in $W^{-\infty}$
gives a well defined nondegenerate pairing
\begin{equation}\label{eq:prepairing}
P\colon
H_N(\CC^N,W^{+\infty};\ZZ)\times H_N(\CC^N,W^{-\infty};\ZZ) \longrightarrow \ZZ.
\end{equation}
\end{rem}

\subsubsection{The state space of $(W,G)$}
In our setup the above facts can be used to
define the state space as the space of Lefschetz thimbles for
the \emph{stack-theoretic} map
$$W\colon [\CC^N/G]\longrightarrow \CC$$
where $G$ is an $A$-admissible group (\emph{i.e.}~a group of diagonal symmetries
containing $j_W$).
The quasihomogeneity condition yields a state space
naturally equipped with a nondegenerate \emph{inner} pairing.

Let us define the state space first; for the scheme-theoretic  morphism $W\colon \CC^N\to \CC$
it is natural to consider
the relative cohomology $H^*(\CC^N,W^{+\infty})$ which is
concentrated in degree $N$ and  dual to the above space of Lefschetz thimbles.
Since $[\CC^N/G]$ is a stack, and the loci
$W^{+\infty}$ and $W^{-\infty}$ (preimages of $S_M^+$ and
$S_M^-$) are substacks, the suitable
cohomology theory for this setup
is orbifold cohomology (or Chen--Ruan cohomology).
Indeed Chen--Ruan
cohomology admits a natural relative version $$H^{a,b}_{\CR}(U,V)=
\bigoplus_{g}
H^{a-\age(g),b-\age(g)}(U_g,V_g;\CC)^G,$$ where, in
complete analogy with the standard definition of
Chen--Ruan cohomology, the
union runs over the elements of the stabilizers of the stack $U$
(\emph{i.e.}, in our case, the elements of $G$),
the notation $U_g$ and $V_g$ stands for
the stacks where the automorphism $g$ persists,
and $\age(g)$
denotes the age\footnote{We define $\age(\al,V)\in \QQ$
for any finite order
autormorphism $\al$ of a vector space $V$, or --- equivalently
---
for any representation of $\pmmu_r$ for some $r\in \NN$.
Each character $\chi\colon \pmmu_r \to \CC^*$ is of the form
$t \to t^k$ for a unique integer $k$ with $0 \le k \le r - 1$
and, for these representations, we
define the age of $\chi$ as $k/r$.
Since these characters form a
basis for the representation ring of $\pmmu_r$, this extends to a unique additive
homomorphism which we denote by
$\age\colon R\pmmu_r \to \QQ.$} of $g$
acting on the normal bundle of $U_g$
in $U$.
\begin{defn}[$A$ model state space]\label{defn:Astates}
 For any $A$-admissible group $G$, we set
$$\Hcal^{a,b}_{W,G}:=H^{a+q,b+q}_{\CR}([\CC^N/G],W^{+\infty}) \qquad \qquad q=\textsum_jq_j.$$
\end{defn}
\begin{rem}\label{rem:narrowfirst}
The above state space is the direct sum of two spaces: the image and the kernel of
$$i_*\colon H_{\CR}^*([\CC^N/G],W^{+\infty})\to H_{\CR}^*([\CC^N/G])$$
The image of $i_*$ is isomorphic in $\Hcal_{W,G}$
to classes attached to diagonal
symmetries $g$ \emph{fixing only the origin}; these
are \emph{narrow} states (in Section \ref{sect:curves}, Remark \ref{rem:narrow},
we see how these states arise in the geometry of curves and we motivate the 
terminology ``narrow'' from this different viewpoint).
A special case of narrow state is the fundamental class attached to
$j_W$: since $j_W$ fixes only the origin this class is narrow, and --- by
construction --- its degree vanishes. Such a state plays the role of the
unit of $\Hcal_{W,G}$ once the ring structure is set up (see \eqref{eq:quant_prod}).
The complementary space of the space of narrow states,
\emph{i.e.}~the kernel of $i_*$, is  referred to in \cite{FJR1} as
the space of \emph{broad} states. These are classes attached to diagonal
symmetries fixing a nontrivial subspace of $\CC^N$.
\end{rem}

\begin{rem} By making the above definition explicit we may regard
the state space as the direct sum over the elements $g\in G$
of the $G$-invariant cohomology classes
\begin{equation}\label{eq:decompstates}
\Hcal_{W,G}=\bigoplus_{g\in G} H^{N_g}(\CC^N_g,W^{+\infty}_g;\CC)^G,
  \end{equation}
where $N_g$ is the number of coordinates $x_1,\dots,x_N$
fixed by $g$
and $\CC^N_g$ and $W^{+\infty}_g$ denote the
subspaces of $\CC^N$ of $W^{+\infty}$ which are
fixed by $g$. In these terms
narrow states are spanned by the summands satisfying $N_g\neq 0$.
Recall the subspace of $G$-invariant classes within $H^{N_g}(\CC^N_g,W^{+\infty}_g)$
is included in the subspace of $j_W$-invariant classes;
this insures that $H^{N_g}(\CC^N_g,W^{+\infty}_g;\CC)^G$
is equipped with a pure Hodge structure of weight $N_g$; in this way
each class has bidegree $(p,N_g-p)$ in standard cohomology
and, within $\Hcal_{W,G}$  has bidegree $$(\deg_A^+,\deg_A^-)
=(p,N_g-p)+(\age(g),\age(g))-(q,q)
\qquad \qquad (\text{with }q=\textsum q_j).$$
We will usually write $\Hcal_{W,G}^{a,b}$ for the terms of bidegree $(a,b)$
and  $\deg_A$ for the \emph{total} degree $a+b$.
\end{rem}

\subsubsection{The inner pairing}
We now define the nondegenerate inner pairing.
The crucial fact is that the quasihomogeneity of the map $W$ allows us
to define an automorphism
$$I\colon [\CC^N/G]\to [\CC^N/G]$$
exchanging $[W^{+\infty}/G]$ with $[W^{-\infty}/G]$.
Indeed we can set
$$I(x_1,\dots,x_N)=(e^{\pi\cxi q_1},\dots,e^{\pi\cxi q_N}) \qquad \text{for which} \qquad
W(I(x_1,\dots, x_N))=-W(x_1,\dots, x_N).$$
Recall that automorphisms of $[\CC^N/G]$ are defined up to
natural transformation
(composition with elements of $G$).
The automorphism
 $I$ induces the nondegenerate \emph{inner} pairing
\begin{align*}\langle  \cdot, \cdot \rangle \colon
H^N(\CC^N,W^{+\infty};\CC)^G\times H^N(\CC^N,W^{+\infty};\CC)^G &\longrightarrow \CC\\
(\al,\be)\qquad \qquad\qquad \qquad&\mapsto\ \  P(\al,I^*\be)
\end{align*}
via \eqref{eq:prepairing} and passage to cohomology.
Notice that $I$ is defined up to a natural transformation; since we are working
with $G$-invariant cohomology classes this still yields a well defined
pairing.

There is an obvious identification $\varepsilon$ between
$H^{N_g}(\CC^N_g,W^{+\infty}_g;\CC)^G$
and $H^{N_h}(\CC^{N_h},W^{+\infty}_h;\CC)^G$ as soon as $g=h^{-1}$ in $G$.
This allows us to define a nondegenerate pairing between
these two spaces
via $\langle \cdot, \cdot\rangle_g =\langle \cdot, \varepsilon(\cdot)\rangle$
and, in turn, a nondegenerate pairing \emph{globally} on $\Hcal_{W,G}$.
\begin{defn}[pairing for $\Hcal_{W,G}$]\label{defn:Apairing}
We have a nondegenerate inner product
$$\langle \cdot,\cdot\rangle \colon \Hcal_{W,G}\times \Hcal_{W,G}\to \CC$$
pairing $\Hcal_{W,G}^a$ and
$\Hcal^{2\widehat c_W-a}_{W,G}$ for
$$\qquad \qquad\widehat c_W=N-2q=\textsum_j (1-2q_j)
\qquad \qquad \text{(\emph{central charge)}}.$$
\end{defn}
The above formula follows from the well known relation
$\age(g)+\age(g^{-1})=N-N_g$ from Chen--Ruan cohomology
and the overall shift by $q$ in Definition \ref{defn:Astates}; it
shows that the state space behaves like the cohomology of
a variety of complex dimension
$\widehat c_W$. Under the CY condition
$\sum_j q_j=1$, this equals $N-2$; \emph{i.e.},
precisely the dimension of a weighted projective hypersurface,
see Theorem \ref{thm:LGCYstates}.

\subsection{The moduli space} \label{sect:curves}
The relevant moduli space is also defined starting from
the pair $(W,G)$ with $G$ $A$-admissible.
\subsubsection{The moduli stack associated to $W$}
The first step is the definition of a moduli stack $W_{g,n}$
attached to the nondegenerate polynomial
\begin{equation}\label{eq:Wnotn}
W=\sum_{i=1}^s \gamma_i\prod_j x_j^{m_{i,j}}.
 \end{equation}
We provide and elementary definition, simplifying
that of \cite{FJR1} (see Remark \ref{rem:comp}).
The moduli stack $W_{g,n}$
is an \'etale cover of a
compactification of the usual moduli stack
of curves $\mathcal M_{g,n}$.
Set
$$\delta=\exp(\Aut(W));$$
\emph{i.e.}, the
exponent of the group $\Aut(W)$ (the smallest integer $\delta$
for which $g^\delta=1$ for all $\delta\in \Aut(W)$).
A \emph{$\delta$-stable} curve is a proper and geometrically connected
orbifold curve (or twisted curve in the sense of Abramovich and Vistoli)
with finite automorphism group,
stabilizers of order $\delta$
only at the nodes
and at the markings, and trivial stabilizers elsewhere.
The stack $\MMM_{g,n,\delta}$
of $\delta$-stable curves is a smooth and proper Deligne--Mumford stack
which differs only slightly from the usual $\MMM_{g,n}$ (see \cite{chstab}).
The advantage of working
over $\MMM_{g,n,\delta}$ is that $W_{g,n}$ can be regarded as an \'etale and proper cover
of $\MMM_{g,n,\delta}$.
\begin{defn}
 \label{defn:Wcurve}
On a $\delta$-stable curve $C$, a \emph{$W$-structure} is the datum
of $N$ $\delta$th roots
$$\big(L_j\ ,\  \fie_j\colon L_j^{\otimes \delta}\xrightarrow{\ \sim\ }
\omega_{\log}^{\otimes \delta q_j}\big)_{j=1}^N$$
(as many as the variables of
$W$) satisfying the following $s$ conditions
(as many as the monomials
$W_1,\dots, W_s$).
For each $i=1,\dots, s$
and for $W_i(L_1,\dots, L_N)=\bigotimes_{j=1}^N L_j^{\otimes m_{i,j}}$,
the condition
\begin{equation}\label{eq:conditions}
\qquad \qquad W_i(L_1,\dots,L_N)\cong\omega_{\log}
\end{equation}
holds.
A $\delta$-stable curve equipped with a $W$-structure is called an
\emph{$n$-pointed genus-$g$ $W$-curve}. We denote by $W_{g,n}$ their moduli stack
\[ \xymatrix@C=1.1cm{W_{g,n} := \Bigg \{(L_1,\fie_1),\dots, (L_N,\fie_N)\ar
[r]_{\hspace{1.2cm}\text{roots}}&{C\ni \sigma_1,\dots,\sigma_n}\ar[r]_{\hspace{0.80cm}\text{curve}}
& X &W_i(L_1,\dots, L_N)\cong \omega_{\log}\Bigg \}_{/\cong}.\\
}
\]
\end{defn}

\begin{rem}\label{rem:group}
If we replace $\omega_{\log}$ by the trivial line bundle $\Ocal$
in \eqref{eq:conditions}, we obtain a different moduli stack
$W_{g,n}^0$ which also deserves special attention (see Theorem \ref{thm:torsor}).
\end{rem}

\begin{rem} Since $j_W$ is in $\Aut(W)$,
it is automatic that $\delta q_j$ is integer. On the other hand,
the exponent $\delta$ of $\Aut(W)$ is not the order $\abs{j_W}$ of $j_W$.
As a counterexample consider the $D_4$ singularity
$x^3+xy^2$: the order of $j_W$ is $3$ but the exponent $\delta$ is $6$.
\end{rem}

\begin{rem}
It is straightforward to see that $W_{g,n}$ is a proper and \'etale cover
of the proper moduli stack $\MMM_{g,n,\delta}$.
Let us introduce the following notation.
Given an $m$-tuple of line bundles $\vec E=(E_1,\dots,E_m)$ and
an $n\times m$ matrix $A=(a_{i,j})$ we denote by $A\vec E$ the $n$-tuple of line bundles
\begin{equation}\label{eq:matrixnotn}
A\vec E=(\otimes_j E_i^{\otimes a_{i,j}})_{i=1}^n.
 \end{equation}
A similar notation holds for an
$m$-tuple of isomorphisms of line bundles
$\vec f=(f_1,\dots,f_m)\colon \vec E \to \vec F$;
we write $A\vec f$ for the $n$-tuple of isomorphisms of line bundles
$(\otimes_j f_i^{\otimes a_{i,j}})_{i=1}^n$
from $A\vec E$ to $A\vec F$.

With this notation we may rephrase the definition of $W_{g,n}$.
Consider the $N$ roots $(L_j,\fie_j)_j$ as a pair of vectors as above
$(\vec L, \vec\fie)$ and define $E_W$ as the matrix $(m_{i,j})$ from \eqref{eq:Wnotn}.
Then,
$M\vec \fie$ is an $s$-tuple of isomorphisms
$M\vec \fie\colon E_W\vec L^{\otimes \delta}\to E_W(\omega_{\log}^{\otimes \delta q_1},
\dots,\omega_{\log}^{\otimes dq_N})^t$ identifying
the $\delta$th tensor powers
$W_i(L_1,\dots,L_N)^{\otimes  \delta}$ to $\omega_{\log}^{\otimes \delta}$.
Hence, it is automatic that
$W_i(L_1,\dots,L_N)\otimes \omega_{\log}^\vee$ is $\delta$-torsion and
the stack $W_{g,n}$
is merely the open and closed substack where
such a line bundle is actually \emph{trivial}.
This is an open and closed condition within a fibred product of categories
of $\delta$th roots. Since
the stacks of $\delta$th roots of a given
line bundle have been shown in \cite{chstab} to be
proper and \'etale over $\MMM_{g,n,\delta}$, the following
theorem follows.
\end{rem}

\begin{thm}\label{thm:torsor}
Let $W$ be a nondegenerate quasihomogeneous polynomial of
type $(q_1,\dots,q_N)$.

\begin{enumerate}
\item
The stack $W_{g,n}$ is nonempty if and only if
$n>0$ or $\abs{j_W}$ divides $2g-2$.
It is a proper, smooth, $3g-3+n$-dimensional
Deligne--Mumford stack; more precisely, it is
\'etale over $\MMM_{g,n,\delta}$ which is a proper
and smooth stack of dimension $3g-3+n$.

\item The stack $W^{0}_{g,n}$ (see Remark \ref{rem:group})
carries a structure of a group over
the stack of genus-$g$ $n$-pointed
$\delta$-stable curves $\MMM_{g,n,\delta}$
with composition law
\[
\xymatrix
@R=0.3cm
{
W^0_{g,n} \times_{\delta} W^0_{g,n}\ar[r]& W^0_{g,n},
}
\]
where $\times_{\delta}$ denotes the fibred product over $\MMM_{g,n,\delta}$. The degree of
$W^0_{g,n}$ over $\MMM_{g,n,\delta}$ is equal
to $\abs{\Aut(W)}^{2g-1+n}/\delta^N$ for $n>0$ and
$\abs{\Aut(W)}^{2g}/\delta^N$ for $n=0$.
\item
The stack $W_{g,n}$
is a torsor under the group stack ${W}^0_{g,n}$
over $\MMM_{g,n,\delta}$.
In particular, its degree over $\MMM_{g,n,\delta}$ equals that of
$W^0_{g,n}$ over
$\MMM_{g,n,\delta}$.
We have a surjective \'etale
morphism and an action
\[\xymatrix@R=0.3cm
{W_{g,n} \ar[dd]
&& \\
&&W^0_{g,n} \times_{\delta} W_{g,n}\ar[r]& W^c_{g,n}\ \ .\\
\MMM_{g,n,\delta}\ }
\]\qed
\end{enumerate}
\end{thm}

\begin{rem}\label{rem:comp}
The above moduli stack slightly differs from that used in \cite{FJR1}.
In the present paper a point representing a curve with
trivial automorphism group is equipped with $N$ automorphisms acting
by multiplication by $\xi_\delta$ on the fibres of $L_1,\dots,L_N$; therefore, as a
stack-theoretic point
it should be regarded as $B(\pmmu_\delta)^N$. In
\cite{FJR1}, the isomorphisms
$W_i(L_1,\dots, L_N)=\bigotimes_{j=1}^N L_j^{\otimes m_{i,j}}$
are included in the data defining an object; this involves some technicalities
on the compatibility between these isomorphisms \cite[2.1.4]{FJR1}.
Adding these extra data imposes further constraints
to the multiplications by $\delta$th roots of unity
along the fibres; hence, the generic automorphism
group in \cite{FJR1} may be
smaller than $\pmmu_\delta^N$ and is actually equal to $\Aut(W)\subseteq
(\pmmu_\delta)^N$.
It is easy to see that
the moduli functor of \cite{FJR1}
is an \'etale cover of $W_{g,n}$, locally isomorphic
to $$B\Aut(W)\to B(\pmmu_{\delta})^N;$$
therefore, since we regard the relevant classes defined in
\cite{FJR1} as pushforwards to $W_{g,n}$,
this issue does not affect the intersection theory
on the stack.
\end{rem}

\subsubsection{Decomposition of $W_{g,n}$ according to the type of the markings}
Consider a $\delta$th root $L$
of a line bundle pulled back from the universal stable curve of $\ol{\mathcal M}_{g,n}$
(\emph{e.g.}, $\omega_{\log}^{c}$ for some $c$).
An index
\begin{equation}\label{eq:mult}
\mult_{\sigma_i}L=\Theta_i\in [0,1[\end{equation}
is determined by the local index of the
universal $\delta$th root $L$
at the $i$th marking $\sigma_i$.
More explicitly, the local picture of
$L$ over $C$ at the $i$th marking
$\sigma_i$ is
parametrized by the
pairs $(x,\la)\in \CC^2$, where
$x$ varies along the curve and $\la$ varies along the fibres of the
line bundle.
The stabilizer $\pmmu_{\delta}$
at the marking acts as $(x,\la)\mapsto (\exp({2\pi \cxi/\delta})
x, \exp({2\pi \cxi\Theta_i})\la)$.
In this way, the local picture of $L$ provides an explicit definition of
$\Theta_1,\dots,\Theta_n$ for the markings $\sigma_1,\dots, \sigma_n$.
As a consequence of Definition \ref{defn:Wcurve},
the stack $W_{g,n}$ decomposed into several connected components
defined by specifying the multiplicities of  the roots $L_1,\dots,L_N$
at the points $\sigma_1,\dots,\sigma_n$.
We organize these data into $n$ multi-indices ${h}_1,\dots,{h}_n$ each one
with $N$ entries.
\begin{defn}
Let us fix $n$ multi-indices
with $N$ entries ${h}_i=(e^{2\pi\cxi \Theta^i_1},\dots,e^{2\pi\cxi \Theta^i_N})
\in U(1)^N$ for
$i=1,\dots, n$ and $\Theta^i_j\in [0,1[$.
Then
$W({h}_1,\dots,{h}_n)_{g,n}$
is the stack of $n$-pointed genus-$g$ $W$-curves satisfying
the relation
$\Theta^{i}_{j}=\mult_{\sigma_i} L_j,$
where $\Theta^{i}_{j}$ is the $j$th entry of
${h}_i$.
\end{defn}
\begin{pro}\label{pro:decomp}Let $n>0$.
The stack $W_{g,n}$ is the disjoint union
$$W_{g,n}=
\bigsqcup_{{h}_1,\dots,{h}_n \in U(1)^N}
W({h}_1,\dots,{h}_n)_{g,n}.$$
The  stack
$W({h}_1,\dots,{h}_n)_{g,n}$ is nonempty if and only if
\begin{equation}\label{eq:degcond}\begin{cases}
{h}_i=(e^{2\pi\cxi \Theta^i_1},\dots,e^{2\pi\cxi \Theta^i_N})\in \Aut(W)&i=1,\dots,n;\\
q_j(2g-2+n)-\sum_{i=1}^n \Theta^i_j\in \ZZ&j=1,\dots,N.
\end{cases}\end{equation}
\end{pro}

\begin{rem}\label{rem:narrow} A marking of a $W$-curve
is therefore attached with a multi-index
${h}=({h}_1,\dots,{h}_N)\in \Aut(W)$.
The case where all coordinates of ${h}$ are nontrivial
is special: the sections of the line bundles $L_1,\dots,L_N$
necessarily vanish at such a marking. In this sense the bundle at that marking
is ``narrow''
(this may provide a geometric explanation for
the terminology of \S\ref{subsect:states}).
Similarly, a narrow node is a node whose multiplicities ${h}$ and
${h}^{-1}\in \Aut(W)$ on the two branches are narrow in the sense
of \S\ref{subsect:states}. Again sections
necessarily vanish at such a node.
\end{rem}

\subsubsection{The moduli stack associated to $W$ and $G$}
We identify open and closed substacks of
$W_{g,n,G}$ where the local indices ${h}$
only belong to a given subgroup $G$ of $\Aut(W)$. This happens because $G$
can be regarded as
the group of diagonal symmetries
of a polynomial
$$W(x_1,\dots,x_N)+\text{ extra quasihomogeneous terms in the
variables $x_1,\dots, x_N$}.$$
We may allow negative exponents
in the extra terms; we only require that the extra
monomials are distinct from those of $W$ but involve the same variables
$x_1,\dots,x_N$ with
charges $q_1,\dots,q_N$.
The following lemma is due to Krawitz \cite{Kr}.
\begin{lem}[Krawitz \cite{Kr}] For any $A$-admissible subgroup $G$ of $\Aut(W)$, there exists a
Laurent power series $Z$ in the same variables
$x_1, \dots, x_N$ as $W$
such that
$W(x_1,\dots,x_N)+Z(x_1,\dots,x_N)$
is quasihomogeneous in the variables $x_1,\dots,x_N$
with charges $q_1,\dots,q_N$ and we have
$G=G_{W+Z}.$
\qed
\end{lem}
In this way to each $A$-admissible
subgroup $G$ of $\Aut(W)$ we can associate a substack
$W_{g,n,G}$ of $W_{g,n}$ whose object will be
referred to as $(W,G)$-curves.
\begin{defn}
Let
$W_{g,n,G}$ be the full subcategory of $W_{g,n}$ whose
objects $(L_1,\dots,L_N)$ satisfy
$Z_t(L_1,\dots,L_N)\cong \omega_{\log}$,
where $Z=\sum_t Z_t$ is the sum of monomials $Z_t$
satisfying $G=G_{W+Z}$.
\end{defn}
\begin{rem} \label{rem:canemb}
The above definition of
$W_{g,n,G}$ makes sense.
It is immediate that the definition of $W_{g,n}$
extends when $W$ is a
quasihomogeneous power series.
It is also straightforward that the definition of
$W_{g,n,G}$ does not depend on the choice of $Z$.
Assume that there are two polynomials $Z'$ and $Z''$
satisfying $G=G_{W+Z'}=G_{W+Z''}$. We can define a third polynomial $\wt Z$
by summing all distinct monomials of $Z'$ and $Z''$. Then we immediately have
 $G_{W+\wt Z}=G$
and
$$(W+Z'\text{-conditions})_{g,n}\supseteq (W+\wt Z\text{-conditions})_{g,n}\subseteq
(W+Z''\text{-conditions})_{g,n}.$$ Notice that these inclusions
cannot be strict: the fibres  over $\MMM_{g,n,\delta}$ of all the
three moduli stacks involved
are zero-dimensional stacks
all isomorphic to  the disjoint union of $\abs{G}^{2g-1+n}$ copies
of $B(\pmmu_\delta)^N$.
\end{rem}

\begin{rem}
As in Proposition \ref{pro:decomp}, for $n>0$, we have
     $$W_{g,n,G}=
     \bigsqcup_{{h}_1,\dots,{h}_n\in G}
     W({h}_1,\dots,
     {h}_n)_{g,n,G},$$
     where ${h}_i\in G$ is the local index at the
     $i$th marked point.
\end{rem}

\begin{exa}
The case where $G=\Jgr$ is easy to work out.
The substack $W_{g,n,\Jgr}\subseteq W_{g,n}$
is the image of the stack of roots of $\omega$ of
order $\abs{j_W}$ via
the functor
$$(L,\fie)\mapsto \big((L^{\otimes \delta q_1}, \fie^{\otimes dq_1}),
\dots, (L^{\otimes \delta q_N},\fie^{\otimes dq_N})\big)$$
(recall that $\delta$ and $\abs{j_W}$ differ in general).
\end{exa}

\subsubsection{Tautological classes}
The so-called psi classes and kappa classes are defined as
$$\psi_i=\sigma_i^*\omega_{\pi} \quad \text{(for $i=1,\dots, n$)} \qquad \kappa_h=\pi_*(c_1(\omega_{\log})^{h+1})\in H^{2h}(W_{g,n})
\qquad \text{(for $h\ge 0$)},$$
where $\pi$ is the universal curve $\Ccal_{g,n}\to W_{g,n}$ and $\sigma_i$ 
denotes the universal section specifying the $i$th marking.
We can identify each stack $W_{g,n+1}({h}_1,\dots,{h}_n,1)$
to the universal curve $\pi \colon \Ccal \to W_{g,n}({h}_1,\dots,{h}_n)$  and 
express $\kappa_h$ as $\pi_*(\psi_{n+1}^{h+1})$.

Let us consider 
the higher direct images of the universal
$W$-structure $(\mathcal L_1,\dots, \mathcal L_N)$ on the
universal $d$-stable curve $\pi\colon \CCC_{g,n}\to W_{g,n}$.
We express its Chern character in terms of psi classes. 
The normalization 
of the boundary locus parametrizing singular curves in $W_{g,n}$  
can be identified to 
the stack parametrixing pairs ($W$-curves, nodes)
in the universal curve.
We consider the \'etale double cover
$\mathcal D$ given by  
the moduli space of triples ($W$-curves, nodes, a branch of the node).
The stack $\mathcal D$ is naturally equipped with two line bundles
whose fibres are the cotangent lines to the branches; we label the corresponding 
first Chern classes by 
$\psi, \psi'\in H^2(\mathcal D)$ starting from the 
branch attached to the geometric point in $\mathcal D$. 
Recall that  
a $\delta$th root at a node of a $\delta$-stable
curve determines local indices $a, b\in [0,1[$
such that $a+b\in \ZZ$ corresponding to 
the branches of the node (apply for each branch the
definition \eqref{eq:mult} or see \cite[\S2.2]{CZ}). In this way on $\mathcal D$,
the local index attached to the chosen branch determines
a natural
decomposition into open and closed substacks
and natural restriction morphisms of the map to $W_{g,n}$ 
$$\mathcal D=\bigsqcup_{\Theta\in [0,1[}
\mathcal D_\Theta,\qquad
\qquad j_\Theta\colon \mathcal D_\Theta \to W_{g,n}.$$
\begin{pro}\label{pro:GRR}
Let $W$ be a nondegenerate quasihomogeneous polynomial in $N$ variables
whose charges equal $q_1,\dots, q_N$.
For any $j=1,\dots, N$, consider
the higher direct image $R\pi_*\mathcal L_j$
of the $j$th component of the universal $W$-structure.
Let $\ch_h$ be the degree-$2h$ term of
the restriction of the Chern character  to
the stack
$W({h}_1,\dots, {h}_n)_{g,n}$,
where ${h}_i=(e^{2\pi\cxi \Theta^i_{1}},\dots,
e^{2\pi\cxi \Theta^i_{N}})$ for $\Theta^i_j\in [0,1[^N$.
We have
$$\ch_h(R\pi_* \mathcal L_j)=\frac{B_{h+1}
(q_j)}{(h+1)!}\kappa_h-
\sum_{i=1}^n
\frac{B_{h+1}({\Theta^i_{j}})}{(h+1)!}\psi_i^h+
\frac{d}{2} \sum _{0\le \Theta<1}
\frac{B_{h+1}(\Theta)}
{(h+1)!} (j_\Theta)_* \left(\sum_{a+a'=h-1} \psi^a (-\psi')^{a'}\right).$$
\end{pro}
\begin{proof}
This is an immediate consequence of
the main result of \cite{ch}.
\end{proof}

\subsection{The virtual cycle}\label{subsect:virt}
The $\RW$ invariants of $(W,G)$ fit in the
the formalism of Gromov--Witten theory.
     Fix the genus $g$
     and the number of markings $n$ (with $2g-2+n>0$, stability condition);
     then, for any choice of nonnegative
     integers $a_1,\dots, a_n$ (associated to powers of psi classes
     $\psi_1^{a_1}, \dots, \psi_1^{a_n}$)
     and any choice of elements $\al_1,\dots,\al_n\in \Hcal_{W,G}$ we can
     define an invariant (a rational number)
     \begin{equation}\label{eq:FJRWinvariants}
\langle \tau_{a_1}({\al_1}),\dots,\tau_{a_n}({\al_n})
     \rangle^{W,G}_{g,n}.                         \end{equation}
Once the ``target'' $(W,G)$ is fixed, the procedure is similar to Gromov--Witten theory
and shares many features with orbifold theory.
 An intrinsic mathematical
object is attached to each genus $g$ and each number of markings $n$: the so called
``virtual cycle''. Then the psi classes $\psi_1^{a_1}, \dots, \psi_1^{a_n}$
and the state space entries $\al_1,\dots,\al_n\in \Hcal_{W,G}$
naturally yield FJRW invariants via 
intersection theory carried out on 
a moduli space classifying 
the solutions to the Witten equation. 
This is a moduli space overlying 
the moduli space of $W$ curves introduced above. We will not provide a 
treatment at this level of generality, but we identify a number of cases 
where one can reduce to the moduli space of $W$ curves. 

First, let us recall the formalism of 
\cite{FJR1}. There, the definition of \eqref{eq:FJRWinvariants} is
     given by extending linearly the treatment of the special case
     where the entries $\al_i\in \Hcal_{W,G}$ lie
     within a single summand $H^{N_g}(\CC^N_g,W^{+\infty};\CC)$ of
     \eqref{eq:decompstates}. We denote by ${h}_i$ the group element
     satisfying $\al_i\in H^{N_{{h}i}}(\CC^N_{{h}_i},W^{+\infty};\CC)$.
     When all the
markings are \emph{narrow} (in the sense of Remark \ref{rem:narrow}),
a lemma of Witten shows that his
equation has only zero solutions; this provides an heuristic explanation for the
existence of a definition of the FJRW invariants 
in terms of intersection of psi classes against an 
\emph{algebraic} virtual cycle.
In other terms, when $\al_i$ is narrow for all $i$, we have 
     \begin{equation}\label{eq:specifictonarrow}
H^{N_{{h}_i}}(\CC^N_g,W^{+\infty};\CC)\cong {\pmb 1}_{{h}_i}\cdot\CC; 
\end{equation}
\emph{i.e.}~there is a canonical generator ${\pmb 1}_{{h}_i}$ and, by
     abuse of notation, we can write
     $
\langle \tau_{a_1}({{h}_1}),\dots,\tau_{a_n}({{h}_n})
     \rangle^{W,G}_{g,n}
$
and carry the computation directly in the rational cohomology ring 
of $W(h_1,\dots,h_n)_{g,n}$.

In general, for instance for D-type singularities (see \cite{JRbroadD4}), 
nonvanishing invariants attached to \emph{broad} entries 
should  be included in order to define a consistent Gromov--Witten-type theory.
The presence of nonvanishing invariants attached to 
these broad entries can be actually probed indirectly 
via universal relations such as WDVV equation.
This observation is the starting point of the 
FJRW setup. 
Even if \cite{FJR1} provides a coherent setup, a 
direct computation is still an open  problem in general. 
FJRW analytic setup via Witten's PDE indicates that this is due to the lack of 
an effective method to solve Witten's PDE equation. 
(The reader may refer to 
\cite{JRbroadD4} for an example of complete treatment of a D-singularity
involving the broad sector.) 

In this paper, we illustrate the approach
which has allowed a large part of the 
computations available in the literature. Namely, 
we restrict to 
well-behaved cases where we can 
assume that the markings are all within 
the narrow sector. There, 
equation \eqref{eq:specifictonarrow} allows us to focus essentially only on the 
enumerative geometry of the moduli space. 
We will return to the general  case
in the last
part of the section: ``Cohomological field theory in the general case''.

\subsubsection{The case of $A$ singularities. A well-behaved ``concave'' locus.}
In \cite{Wi93b}, Witten considers the case of
the $A_{r-1}$ singularity $W=x^r$. Here, the only $A$-admissible group is
$\Aut(W)=\langle j_W\rangle\cong \pmmu_r$ and the
moduli stack is
$$W_{g,n,G}({h}_1,\dots,{h}_N)=W_{g,n}({h}_1,\dots,{h}_n),$$
for ${h}_i=\exp(e\pi\cxi \Theta_i)$ with $\Theta_i\in [0,1[$.
We have a universal curve
$\pi \colon \Ccal\to W_{g,n}({h}_1,\dots,{h}_n)$ carrying an
$r$th root $\Lcal$ of the relative sheaf
of logarithmic differential
$\omega_{\log,\pi}$.
Let $C$ be a fibre of the universal
curve and let $L$ be the $r$th root
on $C$; consider the space $$V=H^1(C,L).$$
Generically, and away from a few low genus cases, this
vector space has dimension
\begin{equation}\label{eq:DimCycle}
D=(1-g)\left(1-\frac 2 r\right) + \sum_{i=1}^N\left(\Theta_i-\frac1r\right).
\end{equation}
This fails precisely when $H^0(C, L)$ is non-zero. When $H^0(C,L)$ is trivial
we say that $L$ is concave,
the vector space $V$ vary as a complex
vector bundle over $W_{g,n}({h}_1,\dots,{h}_N)$
of rank $D$, the locally free sheaf $R^1\pi_* \Lcal$.
In this case, the virtual cycle is
Poincar\'e dual to the top Chern class of $(R^1\pi_*\Lcal)^{\vee}$
\begin{equation}\label{eq:ctopvirt}
[W_{g,n}({h}_1,\dots,{h}_N)]^{\vir}=\ctop(R^1\pi_*\Lcal)^{\vee}=(-1)^D\ctop (R^1\pi_*\Lcal).
\end{equation}

\subsubsection{Witten's analytic construction for A-singularities.}
Since $H^0(C, L)$ is not trivial in general, Witten suggested the following approach.
At least over the open substack of smooth curves
one has bundles of Hilbert spaces $\Ecal=\Omega^{0,0}(\Lcal)$
and $\Fcal=\Omega^{0,1}(\Lcal)$ (consisting respectively of
$\Lcal$-valued $(0,0)$-forms and $(0,1)$-forms along the fibres of $\pi$),
with a family of operators $\ol\partial\colon \Ecal\to \Fcal$.
Choosing a Hermitian metric on $L$ defines an
    isomorphism $\ol{L}\cong L^\vee$.
In this way the Serre duality (SD) map
\begin{equation}\label{eq:SD}
s\in
H^0(C, L^{\otimes r-1})\hookrightarrow H^0(C, \omega \otimes L^{\vee})
\stackrel{\rm SD}{=} H^1(C,L)^\vee\ni s^{r-1}                                                  \end{equation}
is regarded as a family of maps
on the total space of $\Ecal$
\begin{align}\label{eq:Wittenmap}
\ol {\partial W}\colon \Ecal &\to p^*\Fcal,\\
s&\mapsto \ol s^{r-1},\nonumber
\end{align}
where $p$ is the projection $\Ecal$ to the moduli stack $W_{g,n}$.
Witten considers the section of $p^*\Fcal\to \Ecal$
\begin{equation}\label{eq:WW}
\WW(s)=\ol \partial (s) +\ol {\partial W}(s),  \end{equation}
for which
\begin{equation}\label{eq:nondegWclass}
 \WW(s)=0 \ \ \Leftrightarrow\ \  s=0;
\end{equation}
\emph{i.e.}~$\WW$  vanishes only on the zero section
of the total space of $\Ecal$.
Then, the above data  defines a topological Euler
    class $(-1)^De(\WW: \Ecal \to \pi^*\Fcal)$ which generalizes
\eqref{eq:ctopvirt}.
It was not clear, however, how to extend this
approach to singular curves and to the whole stack $W_{g,n}({h}_1,\dots,{h}_n)$.

\subsubsection{The algebraic counterpart for A-singularities.}
The algebraic counterpart of the above analytic construction has been
provided in \cite{PV} via bivariant intersection theory using
MacPherson's graph construction.
In \cite{chK}, the first author
provided a compatible
construction directly in the $K$ theory ring of
$W_{g,n}$. This can be presented in very simple and explicit terms and
may clarify the above discussion.
Instead of $\Ecal\to \Fcal$, consider a complex of coherent locally free sheaves
\begin{equation}\label{eq:complexdiff}0\to  E\xrightarrow{\ \delta\ } F \to 0
\end{equation}
representing the pushforward $R\pi_*\Lcal$ in the derived category; this
exists because $\pi$ is of relative dimension one. We have $\rk(F)-\rk(E)=-\chi(\Lcal)=D$
by Riemann--Roch.
The case where $H^0(C, L)$ constantly vanishes is the case where
we can choose $E=0$ and define the virtual cycle as $\ctop(F^\vee)$; or,
equivalently, in terms of Chern and Todd characters
via the well known Grothendieck formula
$\ctop(F^\vee)=\ch\left(\textsum_{k}(-1)^k\Lambda^k F\right)\td(F^\vee).
$
Since the Todd class is invertible
it makes sense to define
$\td(F^\vee- E^\vee)=\td(F^\vee)/\td(E^\vee).$
In this way, the difficulty lies
in generalizing the term $\ch\left(\textsum_{k}(-1)^k\Lambda^k F\right)$.
To this effect we need to modify
$(\Lambda^kF)_{k=0}^{\rk F}$,
which should be regarded as a complex with zero differential and
$K$ class $\sum_{k} (-1)^k \Lambda^k F$.
Then, this is generalized by a double graded complex of coherent sheaves
with two differentials
$$\left(C^{h,k}=\Sym^hE\otimes \Lambda^kF,\ \ \ \
\delta\colon C^{h,k}\to C^{h-1,k+1}, \ \ \ \
\partial\colon C^{h,k}\to
C^{h-r+1,k-1}\right),$$
where the Koszul differentials $\delta$ and $\partial$
are induced by \eqref{eq:complexdiff}
and by (\ref{eq:SD}-\ref{eq:Wittenmap}).
The reparametrization $(p,q)=(h+k-rk, h+k)$ transforms them
into horizontal and vertical differentials of bidegree
$(-r,0)$ and $(0,-r)$,
More important, due to
\eqref{eq:nondegWclass} the differentials commute and
the cohomology with
respect to the total differential
$H^i_{\delta+\partial}(C^{\bullet,\bullet})$
vanishes except for a finite number of ranks $i$.
In this way, the Chern character
of
$\sum_i (-1)^iH^i_{\delta+\partial}(C^{\bullet,\bullet})$
is well defined
and we can set
\begin{equation}\label{eq:Ktheorydefn}
[W_{g,n}({h}_1,\dots,{h}_n)]^{\vir}=
\ch\left(\textsum_i (-1)^i H^i_{\delta+\partial}(C^{\bullet,\bullet})\right)\td(F^\vee- E^\vee),
\end{equation}
which satisfies cohomological field theory axioms (a key property is
due to Polishchuk \cite{Po}).

\subsubsection{Fan, Jarvis and Ruan's construction for the narrow sector}
Fan, Jarvis and Ruan extended Witten approach to the
case of a general singularity in full generality. The $W$-structure $L_1,\dots,L_N$ can be assembled into
a single vector bundle  $$E=\oplus_{j=1}^N L_j.$$
When the markings are all narrow,
the codimension of the cycle
    $[W({h}_1,\dots,{h}_n)]_{g,n}^{\vir}$ in
    $W({h}_1,\dots,{h}_n)_{g,n}$ equals
    $-\chi(R\pi_* E)$ for $E=\oplus_{j=1}^N L_j$.
By Riemann--Roch for orbifold curves \cite[Thm.~7.2.1]{AGrV},
for ${h}_i=(e^{2\pi\cxi \Theta_1^i},
    \dots,e^{2\pi\cxi \Theta_N^i})$,
    we can explicitly compute
    \begin{align}\label{eq:RRcomput}
    -\chi(R\pi_*E)&=-\rk(E)(1-g)-\deg(E)+\sum_{i,j}\Theta_j^i \nonumber\\
    &=(g-1)N-\sum_{j=1}^N(2g-2+n)q_i+\sum_{i,j}\Theta_j^i \nonumber\\
    &=(g-1)\sum_{j=1}^N(1-2q_j)+\sum_{i=1}^n\sum_{j=1}^N(\Theta_j^i-q_j) \nonumber\\
    &=(g-1)\wh c_W+ \sum_{i=1}^n(\age({h}_i) -q),\nonumber\\
&=(g-1)\wh c_W+ \frac12\sum_{i=1}^n\deg\left({\pmb 1}_{{h}_i}\right),
    \end{align}
    where
    in the last equality we see the role played in
    FJRW theory by the central charge $\widehat c_W=\sum_j  (1-2q_j)$, the age shift
    the constant $q=\sum_j q_j$, and the grading introduced in \S\ref{subsect:states}.
Note how the above formula specializes to \eqref{eq:DimCycle} for $W=x^r$.

Again, since the universal curve $\pi \colon \Ccal \to W_{g,n}({h}_1,\dots,{h}_n)$
is a flat morphism of relative dimension one the pushforward
$R\pi_*E$ in the derived categories can be represented by a two-terms complex of the
form \eqref{eq:complexdiff}; we get $\delta=\oplus \delta_j \colon E\to p^*F$.
Then, in \cite{FJR2}, Witten's morphism
$\ol {\partial W}\colon s\mapsto \ol s^{r-1}$
is replaced
by the direct sum of all partial derivatives
$\ol{\partial_jW}$ for all $j=1,\dots, N$; we get
$\ol{\partial W}=\oplus_j \ol {\partial W_j}\colon E\to p^*F.$
Set $\WW=\delta+\ol {\partial W} $ as in \eqref{eq:WW}.
The nondegeneracy condition for $\WW$
\eqref{eq:nondegWclass} extends immediately
by the nondegeneracy of the polynomial $W$.
Then in \cite{FJR2} the virtual cycle is defined via a topological Euler class construction
    $$[W_{g,n,G}({h}_1, \cdots, {h}_n)]^{vir}=
(-1)^{-\chi(E)}e(\WW: E \to p^*F)\cap [W_{g,n,G}({h}_1, \cdots, {h}_n)]$$
We highlight two subcases.
\begin{description}
\item[Concavity.]
Suppose that all markings are narrow and
suppose that for every fibre $C$ of the universal curve
$H^0(C, L_j)=0$ for all $j$. Then the virtual
cycle is given by
$$\left[W_{g,n,G}({h}_1, \cdots, {h}_n)\right]^{\vir}=
\ctop\left((R^1\pi_*E)^\vee\right) \cap \left[W_{g,n,G}({h}_1, \cdots,
{h}_n)\right].$$
\item[Index zero.]
\label{ax:wittenmap}
Suppose that the dimension of
$W_{g,n,G}({h}_1, \cdots, {h}_n)$ is zero
and that all  markings are narrow. Furthermore, let us assume that
the $\pi_* E$ and $R^1\pi_* E$ are both vector spaces and share the same rank.
Then the
virtual cycle is just the degree of $\WW\colon E\to p^*F$.
\end{description}

\begin{rem}
When $W$ is of Fermat type  paired
with the group $\langle j\rangle$, the genus-zero theory falls into
the concave case.
Here, the expression of the virtual cycle via the top Chern class
allows explicit computations via the  Grothendiek--Riemann--Roch formula of Proposition \ref{pro:GRR};
this happens because $\ctop=[\exp(\sum_{k>0} s_k \ch_k)]_{\rm top}$ for $s_k=(-1)^{k}(k-1)!$ 
(see for example \cite{ChRu}).
Furthermore, since in this case $W$ is a sum of monomials of the form
$x^r$,
even in higher genus
we can describe the virtual cycle by intersecting cycles
defined as in \eqref{eq:Ktheorydefn}. We refer again to \cite{ChRu} for
more details.
\end{rem}

\subsubsection{Cohomological field theory in the general case}\label{subsect:CohFT}
For sake of completeness we finish the section by presenting the formalism
in the general case, beyond the narrow sector.
In general the virtual cycle is defined for
  $$[W_{g,n,G}({h}_{1}, \cdots, {h}_{n})]^{vir}\in
       H_*(W_{g,n}({h}_1, \cdots, {h}_n), \CC)\otimes \prod_{i=1}^n
       H_{N_{{h}_i}}(\CC^{N_{{h}_i}},
       W^{+\infty}_{{h}_i}, \CC)^G$$
       of  degree
       $$2\left((\wh{c}_W-3)(1-g)+n -\sum_i \left(\age({{h}_i})-q\right)\right)$$
(this is the real dimension $6g-6+2n$ of $W_{g,n}$
minus the same degree already discussed in \eqref{eq:RRcomput}).
One of the main achievements of \cite{FJR1} is the proof
of the fact that this cycle,
satisfies the
axioms of a Gromov-Witten type theory.
These axioms can be summarized in terms of the
abstract notion of cohomological field theory \cite{KM}.
This amounts to show that the virtual cycle
has good factorization properties with respect to
the decomposition of the boundary of the moduli space of stable curves $\MMM_{g,n}$.
To this effect, let us introduce the forgetful morphism
$$f\colon W_{g,n}\to \MMM_{g,n}$$
to the standard Deligne--Mumford space. This map factors through the forgetful
map to $\MMM_{g,n,\delta}$ but fails to be \'etale; nevertheless,
Theorem \ref{thm:torsor} determines the degree of $f$ because
$\MMM_{g,n,\delta}$ has degree $1$ over $\MMM_{g,n}$ (see \cite{chstab}).

\begin{defn}
The operators
$\Lambda^W_{g,n,G} \in \Hom(\Hcal_{W,G}^{\otimes n},
H^*({\MMM}_{g,n}))$ are defined as follows. For each entry
$\al_1,\dots,\al_n\in \Hcal_{W,G}$ assume there is a group element
${h}_i\in G$ satisfying
$\alpha_i\in
H^{N_{{h}_i}}(\CC^N_{{h}_i},W^{+\infty}_{{h}_i};\CC)^G$. Then, we set
 $$
\Lambda^W_{g,n,G}(\alpha_1, \cdots, \alpha_k) :=
\frac{|G|^g}{\deg(f)}
f_*\left(\left[W_{g,n,G}({h}_1, \cdots, {h}_n)\right]^{\vir} \cap
\prod_{i=1}^n \alpha_i \right).$$
We extend the definition linearly to the entire space $\Hcal_{W,G}^{\otimes
n}$.

We provide the abstract framework of cohomological field theory.
Suppose that $H$ is a graded vector space with a nondegenerate
    pairing $\langle\cdot,\cdot  \rangle$ and a degree zero unit $1$. To simplify the
    signs, we assume that $H$ has only even degree elements and the pairing
    is symmetric. (When this is not the case,
there are systematic solutions in terms of cohomological field theories over super-state spaces)
Once and for all, we choose a homogeneous basis
    $\phi_{\alpha}$ ($\alpha=1, \dots, \dim H$) of $H$ with $\phi_1=1$. Let
    $\eta_{\mu\nu}=\langle \phi_{\mu}, \phi_{\nu}\rangle$ and
    $(\eta^{\mu\nu})=(\eta_{\mu\nu})^{-1}$.
    \begin{defn}\label{defn:CohFT}
    A cohomological field theory is a collection of homomorphisms
    $$\Lambda_{g,n}: H^{\otimes n}\rightarrow
    H^*(\overline{\Mcal}_{g,n}, \CC)$$
    satisfying the following properties:

\begin{description}
\item[C1.] The element $\Lambda_{g,n}$ is invariant under the
  action of the symmetric group $S_n$.

\item[C2.] Let $g=g_1+g_2$ and $k=n_1+n_2$ and cosider $\rho_{\rm tree}\colon \MMM_{g_1,n_1+1}\times\MMM_{g_2,n_2+1}\to \MMM_{g,n}$.
Then
$\Lambda_{g,n}$ satisfy the composition property
\begin{multline*} \label{eq:cfttree}
\rho_{\mathrm{tree}}^*
\Lambda_{g_1+g_2,n}(\alpha_1,\dots,\alpha_n)=
\Lambda_{g_1,n_1+1}(\alpha_{i_1},\dots,\alpha_{i_{n_1}},\mu)\,
\eta^{\mu\nu} \otimes \Lambda_{g_2,k_2+1}(\nu
,\alpha_{i_{{n_1}+1}},\dots,\alpha_{i_{n_1+n_2}})
\end{multline*}
for all $\alpha_i \in H$.

\item[C3.] Let
$\rho_{loop}: \overline{\Mcal}_{g-1, n+2}\rightarrow
\overline{\Mcal}_{g, n}$
be the loop-type gluing morphism. Then
\begin{equation}
\label{eq:cftloop}
\rho_{\mathrm{loop}}^*\,\Lambda_{g,n}(\alpha_1,\dots,\alpha_n)\,=\,
\Lambda_{g-1,n+2}\,(\alpha_1,\dots,\alpha_n, \mu,
\nu)\,\eta^{\mu\nu},
\end{equation}
where $\alpha_i$, $\mu$, $\nu$, and $\eta$ are as in C2.

\item[C4a.] For all $\alpha_i$ in $H$ we have
\begin{equation}
  \label{eq:identity}
\Lambda_{g,n+1}(\alpha_1,\dots, \alpha_n, 1) =
\pi^*\Lambda_{g,n}(\alpha_1,\dots, \alpha_n),
\end{equation}
where $\pi:\overline{\Mcal}_{g,n+1} \rightarrow
\overline{\Mcal}_{g,n}$ is the forgetful morphism.
\item[C4b.] We have
\begin{equation}
\label{eq:identity2}
\int_{\overline{\Mcal}_{0,3}}\,\Lambda_{0,3}(\alpha_1,\alpha_2,1)
= \langle\alpha_1,\alpha_2\rangle.
\end{equation}
\end{description}
\end{defn}
For each cohomological field theory, we can generalize the notion
of intersection number, the generating function and total
descendant potential function. Let
$$\langle \tau_{a_1, \alpha_1}, \cdots, \tau_{a_n,
\alpha_n}\rangle^{\Lambda}_g=
\int_{\overline{\Mcal}_{g,n}}\prod_i \psi^{a_i}_i\Lambda_{g,n}(\phi_{\alpha_1},
\dots, \phi_{\alpha_n}).$$
By associating a formal variable $t^{\alpha}_i$ to $\tau_{i,
\alpha}$, we define generating functions
$$\Fcal_{\Lambda}^g=\sum_{n\geq 0}\frac{t^{\alpha_1}_{a_1}\dots
t^{\alpha_k}_{a_n}}{n!}\langle\tau_{a_1, \alpha_1}, \dots,
\tau_{a_n, \alpha_n}\rangle_g$$
    and their total  potential function
    $$\Dcal_{\Lambda}=\exp\left(\textsum_{g\geq
    0}h^{g-1}\Fcal^g_{\Lambda}\right).$$
\end{defn}
\begin{thm}[Fan--Jarvis--Ruan \cite{FJR1}]
\label{thm:CohFT}Let $1$ be the distinguished generator ${\pmb 1}_{j_W}$ attached
$j_W$ lying in the $A$-admissible group $G$. Let
$\langle\cdot ,\cdot\rangle^{W,G}$
denote the pairing on the state space
$\Hcal_{W,G}$. Then, the collection $(\Hcal_{W, G}, \langle\cdot ,\cdot \rangle^{W,G},
\{\Lambda^W_{g,n,G}\}, 1)$ is a cohomological field theory.\end{thm}

The following properties hold. \begin{enumerate}
                                \item {Decomposition.}
If $W_1$ and $W_2$ are two singularities in distinct
variables, then the cohomological field theory arising from $(W_1+W_2,G_1\times G_2)$ is
the tensor product of the cohomological field theories arising
from $(W_1,G_1)$ and $(W_2,G_2)$.

\item {Deformation invariance.}
Suppose that $W_t, t\in [0,1]$ is a one-parameter family of nondegenerate polynomials such that $W_t$ is
    $G$-invariant. Then, we have a
    canonical isomorphism  $\Hcal_{W_0, G}\cong \Hcal_{W_1, G}$.
    Under the above isomorphism,
    $$\Lambda^{W_0}_{g,n,G}=\Lambda^{W_1}_{g,n,G}.$$
    Namely, $\Lambda^W_{g,n}$ depends only on $(q_1, \dots, q_N)$ and on $G$.
    Note also that, when applied to a deformation of a polynomial $W$ along a loop,
    this property implies monodromy invariance for  $\Lambda^W_{g,n,G}$.
    \item{$\Aut(W)$-invariance.}
    The group $\Aut(W)$ acts on $\Hcal_{W,G}$
    in an obvious way.
    Then, $\Lambda^W_{g,n,G}$ is invariant with respect to the action of
    $\Aut(W)$ on each state space entry $\al_1,\dots,\al_n\in \Hcal_{W,G}$.
    \end{enumerate}

\section{State spaces: a complete picture}\label{sect:classMS}
State spaces are the
cornerstones of Gromov--Witten theory and of Fan--Jarvis--Ruan--Witten theory.
At their level,
we can provide an exhaustive picture featuring LG-CY correspondence as well as
mirror symmetry. To this effect, since we already went through $A$ model state spaces,
let us introduce the $B$ model state space.

\subsection{${\it B}$ model state space}
The present discussion parallels the above introduction of the $A$ model state space.

\subsubsection{Local algebra from the classical point of view}
Consider the  {local algebra} (also
known as the  {chiral ring} or the  {Milnor ring})
$
\Qcal_W:=\CC[x_1, \dots, x_N]/\Jac(W)$, with $\Jac(W)$ equal to the
Jacobian ideal generated by partial derivatives
$\Jac(W)=\left({\partial_1 W}, \dots, {\partial_N W}\right).$
We regard each polynomial $\alpha(x_1,\dots, x_N)$ in $\Qcal_W$ as an 
$N$-form 
$\al(x_1,\dots, x_N)dx_1\wedge\dots\wedge dx_N$. 
In this way a diagonal symmetry $\Diag(\la_1,\dots,\la_N)$ operates on $\prod_j x_j^{m_j} dx_1\wedge\dots\wedge dx_N$ 
by multiplication by $\prod_{  j} \la_j^{m_j+1}$. 
 
The
local algebra is graded by assigning 
$x_j\mapsto q_j$; in this way $\prod_j x_j^{m_j} dx_1\wedge\dots\wedge dx_N$
has degree $\sum_j (m_j+1)q_j$.  
There is a unique element
$$\hess(W)=\det(\partial_i\partial_j W)$$
whose degree is maximal.
The dimension of the local algebra is given by the formula
$$\mu(W)=\prod_i\left(\frac{1}{q_i}-1\right).$$
For $f, g \in \Qcal_W$,
the residue pairing $\langle f, g\rangle$ is determined by writing
$fg$ in the form
$$fg ={\langle f, g\rangle}\frac{\hess(W )}{\mu(W)} + \text{terms of lower degree}.$$
This pairing is well defined and nondegenerate. It endows the local algebra
with the structure of a Frobenius algebra (\emph{i.e.}~$\langle f g, h\rangle =\langle f, gh\rangle$). For more details, see \cite{AGV}.

\subsubsection{The state space of $(W,G)$}
From the modern point of view, the local algebra is regarded as 
a part of the $B$ model theory of singularities. For its application,
it is important to orbifold the construction by $G$. The
orbifold $B$ model graded vector space with pairing $\Ocal_{W,G}$
was essentially worked out by the physicists Intriligator and Vafa
\cite{IV} (see \cite{KA1} for a mathematical account). The
ring structure was constructed later
by Kaufmann \cite{KA2} and Krawitz \cite{Kr}
in the case of the so called ``invertible'' $W$ and $B$-admissible
group $G\subseteq \Aut(W)$.

For each $g\in G$, we write as usual $\CC^N_{g}$ for the points
of $\CC^N$ fixed by  $g$. We write $W_{g}$ for the restriction
of $W$ to ${\CC^N_{g}}$. In this way $W_g$ is
a quasihomogeneous singularity in a subspace of $\CC^N$ and admits a
local algebra $\Qcal_{W_g}$ with a natural $G$-action.
\begin{defn}[$B$ model state space] \label{defn:Bstates}
 For any $B$-admissible group $G$, we set
$$\Qcal_{W,G}=\bigoplus_{g\in G} (\Qcal_{W_g})^G,$$
where $(\ \ )^G$ denots the $G$-invariant subspace.
\end{defn}
\begin{rem}
The state space $\Qcal_{W,G}$  is clearly
    a module over $(\Qcal_W)^G$.
\end{rem}
\begin{rem}\label{rem:Bgrading}
Remark \ref{rem:Gspace} may be regarded as saying that
the $B$ model state space is
--- by construction ---
isomorphic to the
$A$ state space of Definition \ref{defn:Astates}. On the other hand,
the space is equipped with a different Hodge bigrading as follows.
For a $G$-invariant form $\al$ of degree $p$ in $\Qcal_{W_g}$,
the bidegree $(\deg^+_B(\al),\deg_B^-(\al))$ is defined as follows
$$(\deg^+_B(\al),\deg_B^-(\al))=(p,p)+(\age(g),\age(g^{-1}))-(q,q)
\qquad \qquad (\text{with }q=\textsum q_j).$$
We will usually write $\Qcal_{W,G}^{a,b}$ for the terms of bidegree $(a=\deg^+_B(\al),b=\deg^-_B(\al)) $. 
We write $\Qcal_{W,G}^d$ for the terms whose total degree $\deg_B=a+b$ equals $d$.
\end{rem}
\subsubsection{The inner pairing}
Notice that $\Qcal_{g}$ is canonically isomorphic to
    $\Qcal_{g^{-1}}$. The pairing of $\Qcal_{W,
    G}$ is the direct sum of residue pairings
    $$\langle \cdot ,\cdot \rangle\colon
    \Qcal_{g}\otimes \Qcal_{g^{-1}}\rightarrow
    \CC$$
    via the pairing of the local algebra.
\begin{defn}[pairing for $\Qcal_{W,G}$]\label{defn:Bpairing}
We have a nondegenerate inner product
$$\langle \cdot,\cdot\rangle \colon \Qcal_{W,G}\times \Qcal_{W,G}\to \CC.$$
\end{defn}
The pairing descends to a nondegenerate pairing $\Qcal_{W,G}^a\times 
\Qcal^{2\widehat c_W-a}_{W,G}\to \CC$ where 
 $\widehat{c}_W$ is the so-called central charge $\sum_j(1-2q_j)$. 
This quantity,  already appearing in
Definition \ref{defn:Apairing}, plays a fundamental role in singularity theory: the singularities whose central charge $c_W$ 
is less than $1$ are called simple singularities and admit an ADE classification.


\subsection{Mirror symmetry between LG models}
Berglund and H\"ubsch \cite{BH}
consider
polynomials in $N$ variables having $N$ monomials
\begin{equation}\label{eq:invertible}
W(x_1,\dots,x_N)=\sum_{i=1}^N\prod_{j=1}^N x_j^{m_{i,j}}.
\end{equation}
Note that each of the $N$ monomials
has coefficient one; in fact,
since the number of variables equals the number of monomials and $W$ is nondegenerate,
any polynomial of the form
$\sum_{i=1}^N \gamma_i \prod_{j=1}^N x_j^{m_{i,j}}$ can be
reduced to
the above expression by conveniently
rescaling the $N$ variables.
In this way assigning a polynomial $W$ as above amounts
to specifying its exponent square matrix
$$E_W=(m_{i,j})_{1\le i,j\le N}.$$

The polynomials studied in \cite{BH} are called ``invertible'' because the matrix
$E_W$ is an invertible $N\times N$ matrix as a consequence of the
uniqueness of the charges $q_1,\dots, q_N$ (nondegeneracy of $W$).
There is a strikingly simple classification of invertible
nondegenerate singularities by Kreuzer and Skarke \cite{KS}.

An invertible potential $W$ is nondegenerate if and only
if it can be written, for a suitable permutation of the variables, as a
sum of invertible potentials (with disjoint sets of variables)
of one of the following
three types:
\begin{eqnarray}
&W_\text{Fermat} = x^a.\label{eq:Ftype}\\
& W_\text{loop}= x_1^{a_1}x_2+x_2^{a_2}x_3+\dots +x_{N-1}^{a_{N-1}}x_N+x_N^{a_N}x_1.\label{eq:Ltype}\\
&W_\text{chain}= x_1^{a_1}x_2+x_2^{a_2}x_3+\dots +x_{N-1}^{a_{N-1}}x_N+x_N^{a_N}.\label{eq:Ctype}
\end{eqnarray}

One can compute the charges $q_1,\dots,q_N$ by simply setting
\begin{equation}\label{eq:chargesfromM}q_i=\textsum_j m^{i,j},
\end{equation}
the sum of the entries on the $i$th line of $E_W^{-1}=(m^{i,j})_{1\le i,j\le N}$.

Each column $(m^{1,j},\dots,m^{N,j})$
of the matrix $E_W^{-1}$ can be used to
define the diagonal matrix
\begin{equation}\label{eq:colsym}
\rho_j=\Diag(\exp(2\pi \cxi m^{1,j}), \dots,\exp(2\pi \cxi m^{N,j})).
\end{equation}
In fact these matrices satisfy the following properties
$\rho_j^*W=W$; \emph{i.e.}~$W$ is invariant with respect to
$\rho_j$.
Furthermore the group
$\Aut(W)$ of diagonal matrices $\al$ such that $\al^*W=W$ is generated
by the elements $\rho_1,\dots,\rho_N$:
$$\Aut(W):=\{\al={\rm Diag}(\al_1,\dots,\al_N)\mid \al^*W=W\}=
\langle \rho_1,\dots,\rho_N\rangle.$$
For instance, the above mentioned matrix $j_W$ whose diagonal entries are $\exp(2\pi\cxi q_1),\dots,$ and
$\exp(2\pi\cxi q_N)$
lies in $\Aut(W)$ and is indeed the product $\rho_1\cdots\rho_N$.
Recall that $$SL_W=\Aut(W)\cap SL(\CC^N),$$
the matrices with determinant $1$;
in Berglund and H\"ubsch's construction we consider groups $G$
containing $j_W$ ($A$-admissible) and included in $SL_W$ ($B$-admissible). We write $\wt G$ for the
quotient $G/\langle j_W\rangle$.

The geometric side of the LG-CY correspondence
is an orbifold or smooth Deligne--Mumford stack.
More precisely, let $d$ be the least common denominator
of $q_1=w_1/d, \dots, q_N=w_N/d$ (\emph{i.e.}~$d=\abs{j_W}$). Then $X_W=\{W=0\}$
is a degree $d$ hypersurface of the weighted
projective space $\PP(w_1, \dots, w_N)$.
Then, $W$ is nondegenerate
(\emph{i.e.}~$W$ has
a single critical point at the origin) if and only if
$X_W$ is a smooth Deligne-Mumford stack.
Let $W$ be a nondegenerate invertible potential
of charges $q_1,\dots, q_N$ satisfying the \emph{Calabi--Yau condition}
\begin{equation}\label{eq:CYcond}\textsum_j q_j=1.
\end{equation}
The geometrical meaning of this condition is that
$X_W=\{W=0\}$ is of Calabi--Yau type in the sense that the canonical line bundle $\omega$ is trivial
(adjuction formula: $d=\sum_j w_j$).
Under the CY condition, let us point out a special case where several simplifications occur; namely, the case where $w_j$ divides $d=\sum_j w_j$. 
Then the CY hypersurface defined $X_W$ is embedded within a weighted projective stack whose coarse space is Gorenstein. This is a very 
special condition which allows direct computations of  GW invariants in genus zero. 

\subsubsection{The polynomial $W^{\vee}$}
Following Berglund--H\"ubsch, we consider the transposed polynomial $W^{\vee}$ defined by the property
$$E_{W^{\vee}}=(E_W)^{\vee}.$$
Namely, we set 
\begin{equation}\label{eq:daulpoly}
W^{\vee}(x_1,\dots,x_N)=\sum_{i=1}^N\prod_{j=1}^N x_j^{m_{j,i}}
\end{equation}
by
transposing the matrix $(m_{i,j})$ encoding the exponents.
This construction respects the above classifications  \eqref{eq:Ftype}, \eqref{eq:Ltype} and \eqref{eq:Ctype}. 
As a consequence, $W^{\vee}$ is nondegenerate if and only if $W$ is nondegenerate.
Recall that $q_j$ is the sum of the $j$th column of the inverse matrix $E^{-1}_W$.
Hence,
the charges $\ol{q}_1,\dots,\ol q_N $ of $W^{\vee}$ are the sums of the
rows of $E^{-1}_W$. Therefore,
$$\textsum_j q_j=\textsum_j \ol{q}_j.$$
In this way, $W^{\vee}$ is of Calabi--Yau type if and only if  $W$ is of Calabi--Yau type.

The striking idea of
Berglund and Hübsch is that $W$ and $W^{\vee}$ should be related by mirror symmetry.
Clearly this is not true in the naive way:
the mirror of a Fermat quintic three-fold $X_W$
is not the quintic itself as one would get by transposing the corresponding exponent
matrix $E_W$. Instead, as already discussed in the introduction,
the mirror $X_W^{{\vee}}$ is the quotient of $X_W$ by the automorphism group
$(\ZZ_5)^3$.
It was already understood by Berglund--H\"ubsch that the correct statement should
read
$$(W,G) \mbox{ mirror to } (W^{\vee}, G^{\vee})$$
for a conveniently defined dual group $G^{\vee}$.
Many examples of dual groups have been constructed in the literature.
The general construction was given only recently by Krawitz \cite{Kr}.

\subsubsection{The group $G^{\vee}$}
The group $G^{\vee}$ is contained in $\Aut(W^{\vee})$. Recall that $\Aut(W^{\vee})$ is spanned
by the diagonal symmetries $\rho_1^{\vee},\dots,\rho_N^{\vee}$ determined by the columns
of $(E_W^{\vee})^{-1}$ as in \eqref{eq:colsym}:
$$\Aut(W^{\vee})=\langle\rho_1^{\vee},\dots, \rho_N^{\vee}\rangle.$$
Then $G^{\vee}$ is the subgroup defined by
\begin{equation}\label{eq:dualgroup}
G^{\vee}=\textstyle{\left \{\prod_{j=1}^N
(\rho^{\vee}_i)^{a_i} \mid \text{ if }\prod_{j=1}^N
x_i^{a_i} \text{ is } G\text{-invariant}\right\}}.
\end{equation}

More explicitly, we express any $g\in G$ as $g=\rho^{k_1}_1\dots \rho^{k_N}_N$
and $h\in G^{\vee}$ as $h=(\rho_1^{\vee})^{l_1}\dots (\rho^{\vee}_N)^{l_N}.$ Then, $G^{\vee}$ is
determined by imposing within $\Aut(W^{\vee})$ the following
conditions  for all  $g=\rho^{k_1}_1\dots \rho^{k_N}_N\in G$
$$\begin{bmatrix}k_1 & \dots & k_N
  \end{bmatrix}
E^{-1}_W \begin{bmatrix}
          l_1\\
          \vdots\\
          l_N
         \end{bmatrix}
\in \ZZ.$$

We
have the following properties: transposition is an involution $(G^{\vee})^{\vee}=G$,
it is inclusion-reversing ($H\subseteq K\Rightarrow H^{\vee}\supseteq K^{\vee}$),
it sends the trivial subgroup of $\Aut(W^{\vee})$
to the total group $\Aut(W)$, and
it exchanges $\langle j_W\rangle$ and $SL_{W^{\vee}}$.

\subsubsection{Mirror symmetry conjectures between LG models} 
Now, we can state two mirror symmetry conjectures.
Here, ``mirror'' means that the $A$ model and  the $B$ model are exchanged.
The first one is the Berglund--H\"ubsch--Krawitz mirror symmetry of the form \emph{${\rm LG}{{\mid}}\reflectbox{\rm LG}$}.

\begin{conj}[mirror symmetry ${\rm LG}{{\mid}}\reflectbox{\rm LG}$]\label{conj:LGtoLG}
Suppose that $W$ is a nondegenerate
invertible polynomial. Then the Landau--Ginzburg models $(W,G)$ and $(W^{\vee}, G^{\vee})$
mirror each other.
\end{conj}

Let $W$ be invertible and of Calabi--Yau type.
We say $G\subseteq \Aut(W)$  is of {\em Calabi--Yau
type} if  $\langle j_W\rangle \subseteq G\subseteq SL_W$
(the fact that $j_W$ is contained
in $SL_W$ follows from the Calabi--Yau condition \eqref{eq:CYcond}).
In this case $\wt{G}=G/\langle j_W\rangle  $ acts on $X_W$
faithfully and the quotient $[X_W/\wt{G}]$
is still an orbifold with trivial canonical bundle (Calabi--Yau type).
The properties listed above for the construction
associating $G^{\vee}$ to $G$ show that
$G$ is of Calabi--Yau type if and only if
$G^{\vee}$ is of Calabi--Yau type.
Then, within the Calabi--Yau category,
we obtain a mirror symmetry conjecture of type \emph{${\rm CY}{{\mid}}\reflectbox{\rm CY}$}.

   \begin{conj}[mirror symmetry ${\rm CY}{{\mid}}\reflectbox{\rm CY}$]\label{conj:CYtoCY}
   Suppose that $W$ and $G$ satisfy the Calabi--Yau condition (automatically the same
   holds for $W^{\vee}$ and $G^{\vee})$. Then
   the stack $[X_W/\wt{G}]$ is the mirror of $[X_{W^{\vee}}/\wt{G^\vee}]$.
   \end{conj}

   \begin{rem} Since we have not given a precise meaning to
   to the notion of mirror,
   the above conjectures should be viewed as a
   guideline instead of a mathematical statement.
   In the next section the above conjectures are turned into
   precise mathematical statements.
   One can regard them as relations in
   terms of state spaces. Then,
   they may be read as follows:
   the $A$ model state space of $(W,G)$ is isomorphic to the $B$ model
   state space to $(W^{\vee},G^{\vee})$. Although  elementary, the claim
   is nontrivial. For example
   it does not fit in Borisov--Batyrev duality of Gorenstein cones \cite{BB}.
   This happens systematically when $W$ is not Fermat as was first
   noted in \cite{CL}.
   It was proven by Krawitz.
      \end{rem}
     \begin{thm}[Krawitz \cite{Kr}]\label{thm:Kra}
     Suppose that $W$ is invertible. Then, there is an isomorphism between bigraded vector spaces
     $$\Hcal_{W, G}\cong \Qcal_{W^{\vee}, G^{\vee}}.$$
     \end{thm}
\begin{rem}\label{rem:Kraproof}
       The isomorphism in the theorem is interesting in its own right.
The basic idea is to exchange monomials with group elements.
As already mentioned in Remark \ref{rem:Gspace},
$\Hcal_{W, G}$ and $\Qcal_{W,G}$ are isomorphic
as a consequence of \cite{OS} and \cite{Wa1}.
This isomorphism, however, does not respect the gradings.
Let us express an element of $\Qcal_{W,G}$
as $\bigwedge x_i^{l_i}dx_i \mid \prod_i \rho^{k_i+1}_i \rangle$ where
      $\bigwedge x_i^{l_i}dx_i $ is fixed by $\prod_i \rho^{k_i+1}_i$.
Here, we use the presentation of an element of $\Aut(W)$
      in terms of the generators $\rho_i$. Then the mirror map
in Krawitz's theorem \cite{Kr} is of the form
      $$\bigwedge_i x_i^{l_i}dx_i \mid \prod_i \rho^{k_i+1}_i \Big\rangle
\longmapsto
\bigwedge_i x_i^{k_i}dx_i \mid \prod_i \rho^{l_i+1}_i \Big\rangle.$$
      The proof uses Kreuzer and Skarke's decomposition of invertible polynomials.
Note that there is no analogue decomposition on the
  CY side. This is the main reason why
the LG side is easier to work with in this case.
On the other hand, as we will discuss in the last part of this section
the LG-CY correspondence sets a connection between
the two conjectures given above.
\end{rem}

   \begin{rem}
   Recently, Borisov has found \cite{Bo}
   a new proof of the theorem above via vertex algebras. This
   approach may actually lead to a unified setup including both
   Berglund--H\"ubsch and Borisov--Batyrev duality.
   \end{rem}

   Beyond state spaces the situation is as follows. On the $A$ model side, we have  rigorous theories,
   $\RW$ theory for the LG model and $\GW$ theory for the CY model. The
   counterpart of these theories for the $B$
   model side is incomplete. The genus-zero theory
   should correspond to a Frobenius manifold structure; however, unless $G$ is a trivial group,
   it appears delicate to define the suitable $G$-orbifold version extending
   the state space $\Qcal_{W,G}=\bigoplus_{g\in G} (\Qcal_{W_g})^G$.
   Due to Kaufmann--Krawitz \cite{KA2, Kr},
   we can provide at least an orbifold Frobenius algebra
   construction; \emph{i.e.}, a ring structure on the state space.

   Suppose that $W$ is invertible.
   We define a product on $\bigoplus_{{h}} \Qcal_{W_{{h}}}$
   and then take $G$-invariants.
    The product has the
    properties
    $$\Qcal_{W_{{h}_1}}\otimes \Qcal_{W_{{h}_2}}\rightarrow
    \Qcal_{W_{{h}_1{h}_2}}.$$
    as well as respecting the $\Qcal_W$-module structure
    in the sense that
    $$\alpha  {\pmb 1}_{g_1}* \beta {\pmb 1}_{g_2}=\alpha\beta {\pmb 1}_{g_1}*
    {\pmb 1}_{g_2},$$
    where $\alpha, \beta\in \Qcal_{W_e}$ and ${\pmb 1}_g$ is the unit in
    the algebra $\Qcal_{W_g}$.

    Let
    $$
    {\pmb 1}_{g_1}* {\pmb 1}_{g_2}=\gamma_{g_1, g_2} {\pmb 1}_{g_1g_2},$$
where
\begin{equation}\label{eq:quant_prod}
\gamma_{g,h}
\frac{\hess (W \mid{\CC^N_g \cap \CC^N_h})}{\mu({W\mid{\CC^N_g \cap \CC^N_h}})}
=
\begin{cases}\frac{\hess (W\mid{\CC^{N}_{gh}})}{\mu({W\mid{\CC^N_{gh}}})} &
\text{if $\CC^N_g\cup \CC^N_h\cup \CC^N_{gh}=\CC^N$}\\
0& \text{otherwise.}
\end{cases}
\end{equation}
(We use the convention that
$\hess (W|{\{0\}})=1$.)

    \begin{thm}[Kaufmann \cite{KA2}, Krawitz \cite{Kr}]
     When $W$ is invertible and $G$ is $B$-admissible, the operation
     $*$ is associative, ${\pmb 1}_e*$ operates as the identity,
     and $*$ respects the $G$-action and the double
     grading.
     Therefore, the space of $G$-invariants
    $\Qcal_{W,G}$ is equipped with a Frobenius algebra structure.
\end{thm}

Since the cohomological field theory attached to FJRW theory in the previous section
automatically yields a Frobenius algebra structure for $\Hcal_{W,G}$,
it is natural to
further interpret Conjecture \ref{conj:LGtoLG} as a statement relating
the Frobenius algebra structure $\Hcal_{W,G}$ on the $A$ side to
the Frobenius algebra structure $\Qcal_{W^{\vee},G^{\vee}}$ on the $B$ side.
Krawitz's checked that his vector space isomorphism
for the case $G=\Aut(W)$
$$\Hcal_{W,\Aut(W)}\cong \Qcal_{W^{\vee},(e)}$$
respects the Frobenius algebra structure.
He also provided evidence for the same statement for
$G\subseteq SL_W$ and $W$ of loop type and for other special cases related to
Arnold's strange duality.
We refer to \cite{Kr} for precise statements.

\begin{rem}We can regard these isomorphisms of Frobenius algebra structures as evidence
for an isomorphism between Frobenius manifolds attached to
$(W,G)$ on the $A$ side and to $(W^{\vee},G^{\vee})$ on the $B$-side.
On the other hand we point out again,
that --- unless the group is trivial --- the notion of Frobenius manifold
for pairs of the form $(W,G)$ still lacks a rigorous definition. The problem consists in
orbifolding the Frobenius manifold structure that can be already
defined over the deformation spaces ${\rm Def}(W)$.
As far as we know, the same issue
arises for the $B$ model of Calabi--Yau varieties as soon as
they are equipped with a nontrivial orbifold structure.

\begin{prob}\label{prob:orbifolding}
Orbifold the Frobenius manifold ${\rm Def}(W)$ as well
as the Calabi--Yau $B$ model for $X_W$ and prove a higher genus version of 
the LG-CY correspondence between them.
\end{prob}

This may well lead to a $B$ model version of the
LG-CY correspondence.
\end{rem}

\subsection{LG-CY correspondence}
Let us consider both $A$ model state spaces of FJRW theory and of GW theory.
    The simplest conjecture from the LG-CY correspondence is
    the following {\em cohomological LG-CY correspondence conjecture}.

    \begin{conj}\label{conj:cohomLGCY} Suppose that the pair $(W, G)$ is of Calabi--Yau
    type; \emph{i.e.}~$W$ is nondegenerate (not necessarily invertible) with
    $\sum_j q_j=1$ and $G$ contains $\langle j_W\rangle $ and lies in $SL_W$.
    Then, there is a bigraded
    vector space isomorphism
    \begin{equation}\label{eq:cohomLGCY}
\Hcal_{W, G}^{*,*}\cong H^{*,*}_{\CR}\left([X_W/\wt{G}]; \CC\right),
    \end{equation}
    where the right-hand side is
    Chen--Ruan orbifold cohomology of the stack $[X_W/\wt{G}]$ with $\wt G=G/\langle j_W\rangle$.
    \end{conj}

    This conjecture is certainly not true without assuming that $W$ is of Calabi--Yau type.
    For instance a quartic polynomial in five variables provides an immediate
    counterexample. The Calabi--Yau condition plays a crucial role in the proof of the
    correspondence. In physics, it reflects a to a supersymmetry condition which
    is the source of the physical LG-CY correspondence.
    Even if
    the formula in the statement above makes sense even for
    $G\not \subseteq SL_W$, this
    indicates that the isomorphism may fail without imposing a
    Calabi--Yau condition to $G$.
    Surprisingly the authors found that the above statement still holds
    when $G$ is not contained in $SL_W$.
    We will return to this observation  in the end of the paper where
    we present a higher genus correspondence
    holding precisely for $G\not\subseteq SL_W$.
    \begin{thm} \label{thm:LGCYstates}
    Suppose that $W$ is of Calabi--Yau type and
    that $G$ contains $j_W$ (no upper bound for $G$).
    Then the above cohomological LG-CY correspondence holds.
    \end{thm}

    The main application is the folllowing \emph{classical mirror symmetry}, which is a
    direct consequence of
    the cohomological LG-CY correspondence and
    Krawitz's mirror symmetry theorem of type ${\rm LG}{{\mid}}\reflectbox{\rm LG}$.

\begin{cor}\label{cor:MS}
Suppose that $W$ is invertible and that
the pair $(W, G)$ is of Calabi--Yau type as in Conjecture \ref{conj:cohomLGCY}.
Automatically, also the pair $(W^{\vee}, G^{\vee})$ is of Calabi--Yau type.
Furthermore, the Calabi--Yau orbifolds $[X_W/\wt G]$ and
$[X_{W^{\vee}}/\wt G^{\vee}]$ form a mirror pair in the classical sense;
\emph{i.e.}~we have the following isomorphism between Chen--Ruan cohomology groups
$$H_{\CR}^{p,q}\left([X_W/\wt G];\CC\right)\cong H_{\CR}^{N-2-p,q}\left([X_{W^{\vee}}/\wt G^{\vee}];\CC\right).$$
\end{cor}

\begin{cor}\label{cor:coarsecohom}
Assume that the quotient schemes $X_W/\wt G$ and
$X_{W^{\vee}}/\wt{G^\vee}$ admit crepant resolutions $Z$
and $Z^{\vee}$. Then the above statement yields a statement in
ordinary cohomology:
$$h^{p,q}(Z;\CC)=h^{N-2-p,q}(Z^{\vee};\CC).$$
\end{cor}

In the case where $w_j$ divides $d$,
Corollary \ref{cor:MS} can be deduced from Borisov and Batyrev's
construction of mirror pairs in toric geometry \cite{BB}. As already mentioned,
the general
case does not fit into polar duality because
the associated toric variety is not reflexive.
The following example illustrates this well.

\begin{exa}
We consider the quintic
hypersurface in $\PP^4$ defined as the vanishing locus of
$$W=x_1^4x_2+x_2^4x_3+x_3^4x_4+x_4^4x_5+x_5^5.$$
This is a chain-type Calabi--Yau variety $X$ whose Hodge diamond is clearly equal to
that of the Fermat quintic and is well known: $h^{1,1}=1$, $h^{0,3}=1$,
$h^{1,2}=101$.
The mirror Calabi--Yau
is given by
the vanishing of the polynomial
$$W^{\vee}(x_1,x_2,x_3,x_4,x_5)=x_1^4+x_1x_2^4+x_2x_3^4+x_3x_4^4+x_4x_5^5,$$
which may be regarded as defining a degree-$256$
hypersurface $X^{\vee}$ inside
$\PP(64,48,52,51,41)$. This is a degree-$256$ hypersurface of
Calabi--Yau type
($256$ is indeed the sum of the weights).
In this case, the ambient weighted projective stack is no longer Gorenstein
(all weights but $64$ do not divide the total weight $256$).
Note that the group $SL$ coincides with $\langle
j\rangle$ on both sides; therefore, Corollary \ref{cor:MS} reads
$$h^{p,q}_{\CR}(X;\CC)=h^{3-p,q}_{\CR}(X^{\vee};\CC).$$
Indeed, the Hodge diamond of $X_{W^{\vee}}$ satisfies
$h^{1,1}=101$, $h^{0,3}=1$,
$h^{1,2}=1$ matching \eqref{eq:classMS}.
\end{exa}

   Let us explain the role of the Gorenstein condition.
Let us call the hypersurface $X_W\subset \PP(w_1, \dots, w_N)$
{\em transverse} if the intersection of $X_W$ with every
coordinate subspace of the form $\PP(w_{i_1},
\dots, w_{i_k})$
   is either
   empty or  a hypersurface. The transversality of $X_W$ amounts
   essentially to the ambient space being Gorenstein.
   In another words, if $\PP(w_1, \dots, w_N)$ is not
Gorenstein, $X_W$ will contain some coordinate subspace.
The presence of these coordinate subspaces makes it more difficult
to study $X_W$ and its quotients.
For instance,
it is well known that the enumerative geometry of rational stable maps
for  these coordinate subspaces is an open problem in Gromov--Witten theory
(this is due to the behavior of the virtual fundamental cycle).
   Initially, we thought that nonGorenstein cases such as
loop and chain polynomials may provide counterexamples
   for the classical mirror symmetry conjecture.
   We actually found out that the cohomological LG-CY correspondence
   as well as the classical mirror symmetry conjecture hold
   in full generality.
   Similar issues arise in the enumerative geometry of curves; we will
   discuss them in \S\ref{subsect:CIR}.

\subsubsection{The proof of the cohomological LG-CY correspondence}
To illustrate the idea of the proof, it is instructive
to work out the case of the quintic three-fold.

\begin{exa}\label{exa:isosspace}
    Consider $W=x^5_1+x^5_2+x^5_3+x^5_4+x^5_5$ and the cyclic group
    $G=\langle j{}\rangle$ of order 5.
    For each element $j^{m}=(e^{2\pi\cxi m/5},\dots,e^{2\pi\cxi m/5})\in G$
    with $m=0,\dots,4$
    we compute
    $\Hcal_{W,G}=\bigoplus_{g\in G} H^{N_g}(\CC^{N}_g,W^{+\infty}_g;\CC)^G$
    and the total degree of its elements.

    Let $m\ne 0$ and consider the elements of the summands corresponding to $j^m$.
    These are the narrow states where $H^{N_g}(\CC^{N}_g,W^{+\infty}_g;\CC)^G$ is isomorphic to
    ${\pmb 1}_g \CC$.
    The total degree of $\pmb 1$ is
    $2m-2$.
    We obtain four elements
    of degree $0,2,4$ and $6$; they
    correspond to
    the generators of $H^0(X_W,\CC)$,
    $H^2(X_W,\CC)$, $H^4(X_W,\CC)$ and $H^6(X_W,\CC)$.

    Finally consider the remaining states which are not narrow and lie in
    $H^N(\CC^N_1, W^{+\infty}_1,\CC)^G$. This space is isomorphic
    to the degree-$3$ cohomology group of $X_W$.
    This holds
    in full generality as a consequence of the isomorphism between the
    $G$-invariant part of the local algebra
    and the primitive cohomology.
    The total degree of these elements is $3$.
Therefore, we recover the desired degree-preserving vector space isomorphism.
\end{exa}

    We learn from this example that the dichotomy determined by narrow and broad states
    within the Landau--Ginzburg state space corresponds to the
    well known dichotomy on
    the Calabi--Yau side between fixed classes and variable (or primitive) classes.
    In the orbifold setting, each sector of $X_W$ lies in some
    subweighted projective coordinate space of the form $\PP(w_{i_1},\dots,w_{i_k})$.
    Therefore, this dichotomy
    applies to each sector. We say that an orbifold cohomological
    class is variable (or primitive)
    if it comes from a  variable (or primitive) cohomology class 
    of some sector.
    It is straightforward to match
    the broad  sector with variable classes.
    But it is far trickier to do so for narrow group elements versus
    fixed classes.
    We match these classes via a combinatorial
    construction based on an earlier model for
    Chen--Ruan orbifold cohomology of weighted
    projective spaces due to Boissi\`ere,
    Mann, and Perroni \cite{BMP}.

\section{LG-CY correspondence: towards global mirror symmetry}
\label{sect:genuszeroLGCY}

We state the LG-CY correspondence conjecture at the quantum level; then, we cast it within a mirror symmetry framework.
In the last part of this section we review recent results.

\subsection{LG-CY correspondence between GW theory and FJRW theory}
\label{subsect:LGCY}
The state space isomorphism stated above for
any nondegenerate polynomial $W$ and
any $G\ni j_W$ allows us to extend
the conjecture which we have stated in \cite{ChRu} only
for the quintic polynomial $W$ and
$G=\langle j_W\rangle$. Let us set up Givental's formalism.

\subsubsection{Givental's formalism for GW and FJRW theory}
The setup presented here extends the analogue setup for the quintic presented in \cite{ChRu}.
The genus-zero invariants of both theories are encoded
by two Lagrangian cones, $\Lcal_{\GW}$ and $\Lcal_{\RW}$,
inside two symplectic vector spaces, $(\Vcal_{\GW}, \Omega_{\GW})$ and
$(\Vcal_{\RW},\Omega_{\RW})$.
The two symplectic vector spaces also allow us to state the conjectural
correspondence  in higher genera.
We recall the two settings simultaneously by using the
subscript $\W$, which can be read as ${\GW}$ or ${\RW}$.

We define the vector space $\Vcal_{\W}$ and
its symplectic
form $\Omega_{\W}$.
The elements of $\Vcal_{\W}$
are Laurent series with values
in a state space $H_{\W}$;
$$\Vcal_{\W}=H_{\W}\otimes \CC((z^{-1})).$$
In $\RW$ theory the state space is normally the entire
space $\Hcal_{W,G}$. In $\GW$ theory the state space is
$H_{\CR}([X_W/\wt G];\CC)$.
We choose a basis
$\phi_0,\dots,\phi_k$
for the state space of $\RW$ theory
and a basis $\fie_0,\dots,\fie_k$ for
the state space of $\GW$ theory. We label by zero
${\pmb 1}_{j_W}$ and ${\pmb 1}_{X_W}$, respectively (these two classes play a 
special role
at \eqref{eq:dilatonshift}).
We express the  basis of $H_{\W}$ as
$\Phi_0,\dots, \Phi_k$ and the dual basis $\Phi^0,\dots, \Phi^k$.

The vector space $\Vcal_{\W}$ is equipped with the symplectic form
$$\Omega_{\W}(f_1,f_2)=\Res _{z=0} \langle f_1(-z),f_2(z)\rangle_{\W},$$
where $\langle \ , \ \rangle_{\W}$ is the inner pairing discussed above.
In this way $\Vcal_{\W}$ is polarized as
$\Vcal_{\W}=\Vcal^+_{\W}\oplus \Vcal^-_{\W}$, with
$\Vcal^+_{\W}=H_{\W}\otimes \CC[z]$ and
$\Vcal^-_{\W}=z^{-1}H_{\W}\otimes \CC[[z^{-1}]]$,
and can be regarded as the total cotangent space of
$\Vcal^+_{\W}$. The points of $\Vcal_{\W}$ are parametrized by
Darboux coordinates
$\{q^\iii_{a}, p_{l,j}\}$ and can be written as
$$\sum_{a\ge 0} \sum _{\iii=0}^{k} q^\iii_{a} \Phi_\iii z^a+
\sum _{l\ge 0} \sum_{j=0}^k p_{l,j} \Phi^j (-z)^{-1-l}.$$

We review the definitions of the
potentials
encoding the invariants of the two theories.
In FJRW theory, the invariants are the intersection numbers
\begin{equation}\label{eq:FJRWinv}
\langle \tau_{a_1}(\phi_{i_1}),\dots,\tau_{a_n}(\phi_{i_n})
\rangle^{\RW}_{g,n}= \int_{\MMM_{g,n}} \prod_{i=1}^n \psi_i^{a_i}
\cap \Lambda_{g,n,G}^W(\phi_{i_1},\dots,\phi_{i_n}),
\end{equation}
with $\Lambda_{g,n,G}^W$ as in \S\ref{subsect:CohFT}.
In GW theory, the invariants are the intersection numbers
\begin{equation}\label{eq:GWinv}
\langle \tau_{a_1}(\fie_{\iii_1}),\dots,\tau_{a_n}(\fie_{\iii_n})
\rangle^{\GW}_{g,n,\delta}= \prod_{i=1}^n \ev^*_i
(\fie_{\iii_i})\psi_i^{a_j} \cap [X_W]_{g,n,\delta}^{\vir}. \end{equation} The
generating functions are respectively
\begin{equation}\label{eq:FJRWpot}
\Fcal_{\RW}^g=
\sum_{\substack{a_1,\dots, a_n\\ \iii_1, \dots, \iii_n}}
\langle \tau_{a_1}(\phi_{\iii_1}),\dots,\tau_{a_n}(\phi_{\iii_n})
\rangle^{\RW}_{g,n}\frac{t_{a_1}^{\iii_1}\dots t_{a_n}^{\iii_n}}{n!}
\end{equation}and
\begin{equation}\label{eq:GWpot}
\Fcal_{\GW}^g=\sum_{\substack{a_1,\dots, a_n\\ \iii_1, \dots, \iii_n}}\sum _{\delta\ge 0}
\langle \tau_{a_1}(\fie_{\iii_1}),\dots,\tau_{a_n}(\fie_{\iii_n})
\rangle^{\GW}_{g,n,\delta}\frac{t_{a_1}^{\iii_1}\dots
t_{a_n}^{\iii_n}}{n!}.\end{equation}

In this way, both theories yield a power series
$$\Fcal_{\W}^g=\sum_{\substack{a_1,\dots, a_n\\ \iii_1, \dots, \iii_n}}\sum _{\delta\ge 0}
\langle \tau_{a_1}(\Phi_{\iii_1}),\dots,\tau_{a_n}(\Phi_{\iii_n})
\rangle^{\W}_{g,n,\delta}\frac{t_{a_1}^{\iii_1}\dots t_{a_n}^{\iii_n}}{n!}$$
in the variables $t_{a}^i$ (for FJRW theory the contribution of
the terms $\delta >0$ is set to zero,
whereas $\langle\ \ \rangle^{\W}_{g,n,0}$
should be read as $\langle\ \ \rangle_{g,n}^{\RW}$).

We can also define the partition function
\begin{equation}\label{eq:totpot}
\Dcal_{\W}=\exp\left(\textsum_{g\ge 0} \hbar^{g-1} \Fcal^g_{\W}\right).\end{equation}

Let us focus on the genus-zero potential $\Fcal_{\W}^0$.
The dilaton shift
\begin{equation}\label{eq:dilatonshift}
 q^{\iii}_a=\begin{cases}
t^0_1-1&\text{if $(a,{\iii})=(1,0)$}
\\t^{\iii}_a &\text{otherwise.}\end{cases}
\end{equation}
makes $\Fcal^0_\W$ into a power series in the Darboux coordinates
$q^{\iii}_a$. Now we can define $\Lcal_\W$ as the cone
$$\Lcal_\W:=\{\pmb p=d_{\pmb q}\Fcal^0_\W\}\subset\Vcal_\W.$$
With respect to the symplectic form $\Omega_\W$,
the subvariety $\Lcal_\W$ is
a Lagrangian cone
whose tangent spaces
satisfy the geometric condition $zT=\Lcal_\W\cap T$
at any point
(this happens because both potentials
satisfy the equations SE, DE and TRR of \cite{Givental};
in FJRW theory, this is guaranteed by \cite[Thm.~4.2.8]{FJR1}).

Every point of $\Lcal_{\W}$
can be written as follows
$$-z\Phi_0+ \sum _{\substack{0\le \iii\le k\\a\ge 0}}t^{\iii}_a \Phi_\iii z^a+
\sum_{\substack{n\ge 0\\ \delta\ge 0 }}
\ \sum_{\substack{0\le \iii_1,\dots, \iii_n\le k \\ a_1,\dots,a_n\ge 0}}
\ \sum_{\substack{0\le \epsilon\le k\\ l\ge 0 }}
\frac{t_{a_1}^{\iii_1}\dots t_{a_n}^{\iii_n}}{n!(-z)^{l+1}}
\langle \tau_{a_1}(\Phi_{\iii_1}),
\dots,\tau_{a_n}(\Phi_{\iii_n}),\tau_l(\Phi_\epsilon)
\rangle^{\W}_{0,n+1,\delta}\Phi^\epsilon,$$
where the term $-z\Phi_0$ performs the dilaton shift.
\begin{rem}[$J$-function]\label{rem:Jfunction}
Setting $a$ and $a_i$ to zero, we obtain the points of the form
\begin{equation}\label{eq:Jlocus}
-z\Phi_0+ \sum _{0\le \iii\le k}t_0^\iii \Phi_\iii +
\sum_{\substack{n\ge 0\\\delta\ge 0 }} \ \sum_{0\le \iii_1,\dots,
\iii_n\le k} \ \sum_{\substack{0\le \epsilon\le k\\ l\ge 0
}}\frac{t_{0}^{\iii_1}\dots t_{0}^{\iii_n}}{n!(-z)^{k+1}} \langle
\tau_{0}(\Phi_{\iii_1}),\dots,\tau_{0}(\Phi_{\iii_n}),
\tau_l(\Phi_\epsilon)
\rangle^{\W}_{0,n+1,\delta}\Phi^\epsilon,
\end{equation}
which uniquely determine
the rest of $\Lcal_{\W}$ (via
multiplication by $\exp(\al /z)$ for any
$\al \in \CC$---\emph{i.e.} via
the string equation---and via the divisor equation in GW theory).
We define the $J$-function
$$t=\sum_{\iii=0}^k t_0^\iii \Phi_\iii
\mapsto J_{\W}(t,z)$$ from the state space
$H_{\W}$
to the symplectic vector space
$\Vcal_{\W}$ so that
$J_{\W}(t,-z)$ equals the expression \eqref{eq:Jlocus}.
\end{rem}

\subsubsection{The conjecture}
The following conjecture can be regarded as a geometric version of
the physical LG-CY correspondence
\cite{VW89} \cite{Wi93b}. A
    mathematical conjecture was proposed by the second author in
    \cite{R}.
    The formalism is
    analogous to the
    conjecture of \cite{CIT, CR} on crepant resolutions of
    orbifolds and uses Givental's quantization from \cite{Givental},
    which is naturally defined
    in the above symplectic spaces $\Vcal_{\RW}$ and $\Vcal_{\GW}$.
    In \cite{ChRu} we provided a precise mathematical statement
    for the special case of the quintic three-fold; here, we build upon recent work,
    and provide a general statement applying to all CY orbifolds
    that can be written as hypersurfaces $X_W$ in weighted projective spaces and
    to the finite group quotients $[X_W/\wt G]$.

\begin{conj}[LG-CY correspondence]\label{conj:LGCY}
    Consider the Lagrangian cones $\Lcal_{\RW}$ and $\Lcal_{\GW}$.
    \begin{enumerate}
    \item[(1)] There is a degree-preserving $\CC[z, z^{-1}]$-valued linear symplectic isomorphism
       $${\mathbb U}_{\text{\rm LG-CY}}: \Vcal_{\RW}\rightarrow \Vcal_{\GW}$$ and a choice of analytic
       continuation of $\Lcal_{\RW}$ and $\Lcal_{\GW}$
    such that ${\mathbb U}_{\text{\rm LG-CY}} (\Lcal_{\RW})=\Lcal_{\GW}.$
    \item[(2)] Up to an overall constant and up to a choice of
    analytic continuation, the total potential functions  are related by
    quantization of ${\mathbb U}_{\text{\rm LG-CY}} $; \emph{i.e.}
    $$\Dcal_{\GW}=\widehat{{\mathbb U}}_{\text{\rm LG-CY}}(\Dcal_{\RW}).$$
    \end{enumerate}
\end{conj}

    \begin{rem}
    For the readers familiar with the crepant resolution
    conjecture \cite{CIT, CR}, an important difference here is the lack
    of monodromy condition.\end{rem}

    By \cite{CR}, a direct consequence of the first part of the
    above conjecture is the following
    isomorphism between quantum rings.

\begin{cor}\label{cor:quantumiso}
For an explicit specialization of the variable $q$ determined by
${\mathbb U}_{\text{\rm LG-CY}}$, the quantum ring of $X_W$ is isomorphic to the quantum ring
of the singularity $\{W=0\}$.
\end{cor}

\subsection{Towards global mirror symmetry}
\label{subsect:globalMS}

Here, we cast the above conjecture into a global mirror symmetry
framework. The idea is to extend the results presented in the introduction
(in particular Figure
\ref{fig:picture}).

There are two difficulties. First, Gromov--Witten theory
is largely unknown when entries are taken in the variable (\emph{i.e.}~primitive) cohomology
part of the state space. Only in some cases, such as the quintic threefold,
the whole theory is computable because the invariants associated
to the variable cohomology entries are easy to deal with (see \cite{ChRu}).
Second, much of the theory on the $B$ side has not yet been figured out.
For example, in Problem
\ref{prob:orbifolding} we pointed
out that not much is known beyond the untwisted sector; in nontechnical
terms, this means  that
only the case where $G$ is trivial can be
treated in a straightforward way. Here, we present a solution allowing us to move beyond this case.

\subsubsection{Mirror symmetry between invariant $A$-states and untwisted $B$-states}
As a first approach to both problems,
it is natural to try and single out a subclass of invariants involving only
certain state space entries.
This requires checking that
they really form independent theories (\emph{e.g.} 
checking that they assemble into a cohomological field theory
in the sense of \S\ref{subsect:CohFT}).
While doing so, we found out that the
two difficulties discussed above are mirror to
each other. Namely, we point out
that the isomorphism of Theorem \ref{cor:MS} matches
two naturally defined state subspaces. On the $A$ side we consider 
the untwisted sector; \emph{i.e.} the sector attached to the identity element. 
On the $B$ side we consider a subspace which contains the fixed cohomology: 
the space of
cohomology classes of $\Hcal_{W,G}$ which are left invariant by $\Aut(W)$.
Indeed, Krawitz's Theorem \ref{thm:Kra}
yields in particular the identification
\begin{equation}\label{eq:resttheories}
\left[\Hcal_{W,G}\right]^{\Aut(W)}\cong \left[\Qcal_{W^{\vee},G^{\vee}}\right]_{\rm untwisted}
\qquad(\text{the invariant/untwisted mirror symmetry}).
\end{equation}
The right hand side is the local algebra $\Qcal$ of $W^{\vee}$ invariant under
$G^{\vee}$
$$\left[\Qcal_{W^{\vee},G^{\vee}}\right]_{\rm untwisted}=(\Qcal_{W^{\vee}})^{G^{\vee}}.$$
The left hand side can be regarded via \eqref{eq:decompstates}
$$\left[\Hcal_{W,G}\right]^{\Aut(W)}=
\bigoplus_{g\in G} H^{N_g}(\CC^N_g,W^{+\infty}_g;\CC)^{\Aut(W)},$$
where at each summand we have taken invariants with respect of $\Aut(W)$ instead of $G$.
\begin{rem}
In the special case where $W$ is of Fermat type, imposing $\Aut(W)$-invariance is
the same as restricting to the narrow state space:
$$\left[\Hcal_{W_{\rm Fermat},G}\right]^{\Aut(W_{\rm Fermat})}=
[\Hcal_{W_{\rm Fermat},G}]^{\rm narrow}=
\bigoplus_{g\in G\ \mid\  \CC^N_g=(0)} {\pmb 1}_g \CC.$$
Under the LG-CY cohomological correspondence of Theorem \ref{thm:LGCYstates}
this is the same as eliminating all variable (primitive) classes.
In particular, in the case of the quintic Fermat three-fold
discussed in the introductory section, the relation \eqref{eq:resttheories}
shows that fixed cohomology of $X_W$ (in this case $H^{\rm ev}$)
mirrors the
variable cohomology of $X_W^{\vee}$ (in this case $H^{\rm odd}$).
In general, for these Fermat-type cases,
we have
complete genus-zero computations for Gromov--Witten theory (see \cite{CCLT})
as well as for Fan--Jarvis--Ruan--Witten theory (see \cite{CIR}
and
\S\ref{subsect:CIR}).
\end{rem}

\begin{rem} \label{rem:fixed_vs_Autinv}
More generally, if $W$ cannot be written as a Fermat polynomial,
the charges are not necessarily unitary fractions. In other words the
degree $d$ of the corresponding
hypersurface is not necessarily a multiple of all
weights $w_1,\dots,w_N$. Then,
it may happen that $[\Hcal_{W,G}]^{\Aut(W)}\supsetneq [\Hcal_{W,G}]^{\rm narrow}$.
A well known example is $D_4=x^3+xy^2$ where $ydx\wedge dy\mid e\rangle$ is
$\Aut(D_4)$-invariant
and clearly not narrow (see Rem.~\ref{rem:Kraproof} for the notation).
\end{rem}

We now prove that
the condition of ${\Aut(W)}$-invariance singles out a self-contained
cohomological field theory under the Calabi--Yau condition for $W$ and
$G$.

   \begin{lem}
   Assume $\sum_j q_j=1$ and $\langle j_W\rangle \subseteq G\subseteq \Aut(W)$.
   Consider the $\Aut(W)$-invariant subspace of
   either $\Hcal_{W,G}$ or $H_{\CR}([X_W/\wt{G}]; \CC)$. Suppose that $N\geq 5$.
   Then the virtual cycle $\Lambda^W_{g,n,G}(\alpha_1, \dots, \alpha_k)$, for
   $\Aut(W)$-invariant entries $\al_i$, forms a cohomological
   field theory.
   \end{lem}

   \begin{proof}
   It is enough to check the composition axioms.
   For tree-type gluing  morphisms intervening in Definition \ref{defn:CohFT},
   we encounter cycles of the form
   $$\Lambda^W_{g,n,G}(\alpha_1, \dots, \alpha_{k-1}, \beta)$$
   where $\alpha_i$ denotes an $\Aut(W)$-invariant class for $i=1, \dots, k-1$.
    We easily conclude that the condition  $\Lambda^W_{g,n,G}(\alpha_1, \dots, \alpha_{k-1},
    \beta)\neq 0$ holds  only if $\beta$ is  $\Aut(W)$-invariant.
    Next, a dimension argument allows us to
    verify also the axioms involving
loop-type gluing morphism in Definition \ref{defn:CohFT}.
The loop-type gluing situation only appears in the
    higher genus case: assume $g>0$. We can use the forgeful morphism and
reduce to the cases where no state space entry equals the fundamental class.
Then, each entry has
degree $\ge 2$. This happens because we have $N\geq 5$, hence the
hypersurface $X_W$ has dimension $\ge3$ and its lowest degree odd cohomology lies  above
degree $1$.
On the other hand
the degree of twisted sector entries is bigger than $1$ as a consequence of $G\subseteq SL_W$.
    Hence the state space has no degree $1$ class apart from the fundamental class.
    Therefore, we can assume that $\deg (\alpha_i)\geq 2$.
Then, a simple computation shows
    $$\deg \Lambda^W_{g,n,G}(\alpha_1, \dots, \alpha_k)=(5-N)(g-1)+n-\sum_{i=1}^n \frac{1}2 \deg (\alpha_i)\leq 0.$$
    Therefore, $\rho^*_{loop} \Lambda^W_{g,n,G}=0$.
The right hand side is zero for the same reason.
    \end{proof}

    It is reasonable to specialize Conjecture \ref{conj:LGCY} as follows.

    \begin{conj}\label{conj:restLGCY}
    The LG-CY Correspondence Conjecture \ref{conj:LGCY} holds
     for $\Aut(W)$-invariant theories on both sides.
    \end{conj}

    The above conjecture was proved in genus zero for the quintic three-fold by
     the authors \cite{ChRu}
    and for Fermat hypersurfaces in general and $G=\langle j \rangle$ by
    the authors in collaboration with Iritani
    \cite{CIR} (this is the same as saying that the correspondence
holds for every Calabi--Yau hypersurface within a Gorenstein weighted projective stack).
Indeed the case of the quintic three-fold is special,
   because, there, Conjecture \ref{conj:LGCY} in genus
zero follows from Conjecture \ref{conj:restLGCY}.
Then the proof involves calculating the $J$-function of FJRW theory and
a comparison with the $J$-function of the GW side obtained
    by Coates--Corti--Lee--Tseng \cite{CCLT}. We refer to Remark \ref{rem:strategy}
for discussion ot the proof of Conjecture \ref{conj:restLGCY}.

    The general cases of nonGorenstein hypersurface or
larger groups is uncharted territory in GW theory. The starting
point of the proof in the Gorenstein case is
the  observation that the genus-zero theory is
concave; \emph{i.e.} the virtual cycle can be phrased
as the top Chern class
of a bundle. Then, Grothendiek--Riemann--Roch can be applied (see \S\ref{subsect:virt}).
Beyond the Gorenstein case, we do not have such a general method to
compute invariants.
In this sense, the problem is similar in nature
to the computation of the higher genus GW theory of the quintic. There, the difficulty
is also the lack of concavity. In this genus-zero case it
should be noted that the problem may have a chance to 
be approached via
Givental's theory.
For this reason it is clear how this problem not only represents
an exciting new direction in quantum cohomology, but could also
shed new light on GW theory in higher genera.

\subsubsection{Global mirror symmetry}

Let us set up the $B$ side; \emph{i.e.}, the
analogue of the family of Calabi--Yau
three-folds $X_{W,t}^{\vee}$ parametrized by $t$ in
$\PP^1$ from the introductory section.

It is well known that the genus-zero $B$ model
theory corresponds to a period integral vector,
a fundamental object in classical complex geometry.
Given a nondegenerate quasihomogeneous and invertible
polynomial $W$ in $N$ variables of
charges $q_1,\dots,q_N$ adding up to $1$ (CY condition),
we consider the hypersurface
$\{W^{\vee}=0\}$.
It lies naturally in a weighted projective
stack $\PP(w_1, \dots, w_N)$
for suitable choices of positive integers $w_1,\dots,w_N$.
Consider a group of diagonal symmetries $G$ containing $j_W$ ($A$-admissible)
and
included in $SL_W$ ($B$-admissible); then,
by Corollary \ref{cor:MS},
the orbifold $[\{W^{\vee}=0\}/\wt G^{\vee}]$ is the mirror of the
hypersurface defined by $W$.
We consider complex deformations
of the orbifold $[X_{W^{\vee}}/\wt{G^\vee}]$.
Let us focus on the so called
\emph{marginal deformations}; \emph{i.e.}, let
$M_1, \dots, M_l$ be the monomial generators
of the local algebra $\Qcal_{W^\vee}$ \emph{of degree $1$} and
invariant under $G^\vee$.
Then, consider the family of
hypersurfaces
\begin{equation}\label{eq:family}
H_{\pmb a}=\big\{a_0 W^{\vee}+
\textsum_{i=1}^l a_i M_i=0\big\}\subset \PP(w_1,\dots,w_N).
\end{equation}
On an open subscheme of $\PP^{l}$ we may regard
this as a family of Calabi--Yau orbifolds.
The automorphism group
$\Aut(W^{\vee})$ acts on the family of stacks \eqref{eq:family} and
on the base scheme $\PP^l$.
Let us mod out the cyclic subgroup $\langle j\rangle$ acting trivially
everywhere.
Since the morphism is $\Aut(W^\vee)/\langle j\rangle$-equivariant,
we obtain a morphism between the corresponding quotient stacks.
Furthermore, since $G^{\vee}$ acts trivially on the base scheme,
we get a family of Calabi--Yau orbifolds over an open substack of
$[\PP^l/Z]$, with $Z=\Aut(W^\vee)/G^{\vee}$.

In this way, the $B$ side of mirror symmetry is defined.
It is a family of quotient stacks of the form
$$X_{W,\pmb a }^{\vee}=\left[H_{\pmb a}/\wt{G^{\vee}}\right]$$
parametrized by $\pmb a\in[\PP^l/Z].$ On an open substack of
$[\PP^l/Z]$ the family is fibred in Calabi--Yau orbifolds
(smooth Deligne--Mumford stacks whose canonical line bundle is trivial)
\begin{equation}\label{eq:familyCYorb}
 \xymatrix@R=.1cm{
X_{W,\pmb a }^{\vee} \ar[dd]\ar[rr] && \Xcal^\vee \ar[dd]^\pi\\
&\square & \\
\pmb a \ar[rr] && \Omega
}
\end{equation}

\begin{rem}\label{rem:Zcyclic}
 In the special case wher $G^{\vee}=SL_{W^{\vee}}$ the
group $Z$ is cyclic.
Furthermore, when we start from a Fermat polynomial $W$ of degree $d$,
we have $W^\vee=W$ and $Z=\ZZ_d$.
\end{rem}

Let us define the
vector bundle of primitive cohomology with complex coefficients.
On the point $\pmb a$ of the open substack $\Omega$
of $[\PP^1/Z]$
consider the
vector bundle
$$V\longrightarrow \Omega\subset[\PP^l/Z],$$ whose fibre is dual to
the $\wt{G^\vee}$-invariant
part of the kernel of
$$i_*:H_*(H_{\pmb a};\CC)\rightarrow H_*(\PP(\pmb w);\CC).$$
Indeed $i_*$ is an isomorphism in all degrees up to the middle dimension $N-2$.
In degree $N-2$, we have a surjective morphism and the kernel
is precisely the so called
variable (or primitive) homology of the hypersurface.
More systematically
we can set
$R^{N-2}\pi_*(\CC)\otimes \Ocal,$
where $\pi$ is the family \eqref{eq:familyCYorb}.
Then the primitive cohomology sheaf is the
kernel of the Lefschetz operator
$L\colon R^{N-2}\pi_*(\CC)\otimes \Ocal\longrightarrow
R^{N}\pi_*(\CC)\otimes \Ocal.$

\begin{rem}[the relation with the local algebra]
Fibre by fibre, the $\wt{G^\vee}$-invariant part
may be equivalently regarded as the $G^\vee$-invariant
part of the local algebra of $W^\vee_{\pmb a}=a_0W^\vee+\sum_{i=1}^l a_iM_i$
$$\Qcal_{W^\vee_{\pmb a}}=\CC[x_1,\dots,x_N]/\Jac(W^\vee_{\pmb a}).$$
In particular, for $G=\langle j_W\rangle$,
we have $G^\vee=SL_{W^\vee}$
and $V$ is a rank-$(N-1)$ vector bundle
over a one-dimensional base $\Omega$.
\end{rem}

\begin{rem}[Gauss--Manin connection]
On $V$ there is a flat connection $\nabla$, the Gauss--Manin connection,
   given by the local system of integer cohomology
$H^{N-2}(X_{W,\pmb a}^{\vee}  ;\ZZ) \subset H^{N-2}(X_{W,t}^{\vee}  ;\CC)$.
This may be regarded as follows.
Choose a particular fibre $X_{W,\pmb a}^\vee$
and a  basis of $(N-2)$-cycles $\Gamma_1 , \dots, \Gamma_{\rk V}$
for the primitive homology $\ker(i_*)$.
Since the fibration on the scheme $\Omega$ is locally trivial, a local
trivialization can be used to extend the cycles $\Gamma_1,\dots,\Gamma_{\rk V}$
from the chosen
fibre $X_{W,\pmb a}^\vee$ to cycles
$\Gamma_i(z)$ on nearby fibres $X_{W,\pmb a(z)}^\vee$.
This may rephrased as saying that $V$ has a connection $\nabla$ and
that locally over $\Omega$ we can extend a basis of a fibre of $V$ to
a basis of flat sections.
\end{rem}

\begin{rem}[monodromy]
Since the base $\Omega$ is not contractible, the connection may have
nontrivial monodromy.
We may phrase this explicitly using the above cycles $\Gamma_i(z)$
extending the basis $\Gamma_1,\dots,\Gamma_{\rk V}$ of $\ker(i_*)$ over
$\pmb a$.
Indeed, they are locally
constant in the parameter $z$. However,
transporting $\Gamma_i$ along each closed path
produces a cycle homologous to $T\Gamma_i$
for some linear map $T$ (monodromy operator).

Another approach to this monodromy operator is
provided by period integrals.
Choose an
holomorphic $(n,0)$-form $\omega$ on $X_{W,\pmb a}^\vee$.
We can extend $\omega$ locally on a neighbourhood of $\pmb a $ to a holomorphic $(n,0)$-form
on the family of CY orbifolds.
Then, define  the periods integrals of $\omega$
as $$\omega_1(z)=\int_{\Gamma_1(z)}\omega,
\quad \dots,\quad \omega_{\rk V}=\int_{\Gamma_{\rk V}(z)}\omega.$$
They extend
by analytic continuation to multiple-valued functions on $\Omega$,
transforming according to the same monodromy trasformation $T$
operating on the homology classes of the cycles.
We point out that  we can always rescale $\omega$
by a globally holomorphic function
\begin{equation}\label{eq:rescaling}
 \omega\mapsto f\omega;\end{equation}
so the period integrals are defined up to
rescaling.
\end{rem}

\begin{rem}[the base points $0$ and $\infty$]\label{rem:points}
Now we analyse  two special fibres of the above family.
The prototype case is that of the quintic and more generally that of
a homogeneous Fermat polynomial $W$ of degree $N$ in $N$ variables
paired with the group $\langle j_W\rangle$.
Then, $W=W^\vee$ and $G^\vee=SL_W$,
the base is one-dimensional and parametrized by the
homogeneous coordinates $(a_0,a_1)$:
the monodromy is \emph{maximally unipotent}
around $a_0=0$ and \emph{diagonalizable} around $a_1=0$.
This are the Gepner point $0$ and the large volume complex structure point $\infty$.

In the more general set up, we still focus on
two points named $0$ and $\infty$. We assume that $\prod_jx_j$ is a nonvanishing element
of $\Qcal_{W^\vee}$. This is the case apart from a few degenerate cases, which are
treated for example in \cite{MR}.
Then let us set $M_1=\prod_jx_j$
and
\begin{equation}\label{eq:twopoints}
0=(1,0,0,\dots,0)\qquad\text{ and }\qquad\infty=(0,1,0,\dots,0).
\end{equation}
Consider the overlying
fibres
$X_{W,0}^\vee$ and
$X_{W,\infty}^\vee$.
We expect that $\infty$ is the analogue of
the {\em large complex structure point} in the introductory section,
whereas
$0$ should play the role of the \emph{Gepner point}.
More precisely, we propose the
following conjectural picture.
\end{rem}

\medskip

We conjecture that the period integral at $\infty$
 encode the Gromov--Witten theory of
the CY orbifold $[X_W/\wt G]$.

  \begin{conj}[mirror symmetry ${\rm CY}{{\mid}}\reflectbox{\rm CY}$]\label{conj:MSCYtoCY}
       There is a mirror map matching a neighborhood
       of the origin in
       $H^{1,1}_{\CR}([X_W/\wt{G}])^{\Aut(W)}$
       with a neighborhood of the large complex structure point
       $\infty$.
       The mirror map identifies
       the $J_{\GW}$-function
       to the \emph{period integrals}
       $\omega_1,\dots,\omega_{\rk V}$ around $\infty$
       after a suitable rescaling of the form \eqref{eq:rescaling}.
       \end{conj}

It is generally believed that the maximally
quasiunipotent monodromy of the periods
indicates where the $B$ model can be related to
the Gromov--Witten theory of the
mirror variety. Since, in our construction, the fibre $X^\vee_{W, \infty}$
at $\infty$
is highly singular, it
is delicate to describe the monodromy at infinity.
Whereas for $G=\langle j\rangle$,
the singularities occur only over isolated points,
for general choices of $G$ it may be useful to
study birational modifications of
$[\PP^l/Z]$ and of the overlying Calabi--Yau family.
After suitable birational transformations of this base scheme,
we still expect some analogue of the condition of maximal unipotency.
See Morrison \cite{Morr} and Deligne \cite{De}
for more precise treatments on the relation between mirror symmetry
maximal unipotency at the large complex
structure point.

\medskip

At the special point $0$, the monodromy
is diagonalizable. We expect that the local picture encodes $\RW$ theory for
the potential $W$ with respect to $G$.

       \begin{conj}[mirror symmetry ${\rm LG}{{\mid}}\reflectbox{\rm CY}$]\label{conj:MSLGtoCY}
       There is a mirror map matching a neighborhood of
       the origin of $[\Hcal^{1,1}_{W, G}]^{\Aut(W)}$
       with a neighborhood of the Gepner point
       $0$. Via the mirror map,
       the $J_{\RW}$-function is
       matched to the \emph{period basis} $\omega_1,\dots,\omega_{\rk V}$
       after a suitable rescaling as in \eqref{eq:rescaling}.
       \end{conj}

The mirror symmetry ${\rm CY}{{\mid}}\reflectbox{\rm CY}$ conjecture was proved for the smooth case by
       Givental and Liang--Liu--Yau and for CY hypersurface of Gorenstein weighted
       projective space by Corti--Coates--Lee--Tseng \cite{CCLT}. For further
       generalizations we refer to \cite{CIT} and references therein.

       The mirror symmetry conjecture of type ${\rm LG}{{\mid}}\reflectbox{\rm CY}$ was proved for the quintic
       polynomial in five variables and for the group $\langle j \rangle$ by the authors.  For $(W,\langle j \rangle)$ in the Gorenstein case
       and in the restricted version \ref{conj:restLGCY}, it is proven by
       the authors with Iritani \cite{CIR}.
       We refer to \cite{CIR} for a precise statement.

\begin{rem}[LG-CY correspondence via global mirror symmetry]\label{rem:strategy}
The proof of these conjectures allows us to perform
a parallel transport along $\nabla$ on the $B$ side yielding
the identification
between the LG model and the CY
geometry predicted in the LG-CY correspondence conjecture \ref{conj:LGCY}.
Let us restate the scheme of the argument referring to Figure \ref{fig:picture}
for clarity. The two mirror symmetry conjectures above involve the vertical arrows
in Figure \ref{fig:picture}
on the left hand side
and on the right hand side respectively.
The LG-CY correspondence conjecture \ref{conj:LGCY} connects
the two $A$ model structures on the bottom of the picture.
Mirror symmetry allows us to lift the correspondence to
the upper side of the picture, where we can
use the vector bundle $V$ and its flat connection $\nabla$.
\end{rem}

Indeed,
in this way
we verify the genus-zero part
of Conjecture \ref{conj:LGCY} for the quintic;
in other words we verify the claim
that the parallel transport of a basis of flat sections at $0$ is
related to a basis of flat sections at $\infty$
by a constant linear map
${\mathbb U}_{\text{\rm LG-CY}}$. Since ${\mathbb U}_{\text{\rm LG-CY}}$ is symplectic, the
genus-zero part of Conjecture \ref{conj:LGCY} follows, see \cite{ChRu}.

\begin{thm}  \label{thm:quinticLGCY}
The LG-CY correspondence \ref{conj:LGCY} holds in genus zero and matches the
$\GW$ theory of the quintic three-folds in $\PP^4$ to the $\RW$ theory of the
isolated singularity of the corresponding affine cone.
\end{thm}

\begin{rem}[generalizations]
Technically the parallel transport of
the cycles $\Gamma_i(z)$ is difficult to perform explicitly.
Alternatively, in \cite{ChRu} we study
the period integrals $\omega_i$ and their Picard--Fuch's
equation. Locally, in the case
where $W$ is homogeneous and
$G$ equals $\langle j_W\rangle$,
the period integrals belong to the $(N-1)$-dimensional space
of solutions of the Picard--Fuchs equation (this is due to Griffith's transversality).
Since, for $G=\langle j_W\rangle$ and $G^\vee=SL_W$, this is
precisely the rank of the vector bundle $V$, we can avoid the parallel transport
and carry out analytic continuation of
the periods integrals instead.
This approach admits generalizations in the quasi-homogeneous setup; indeed,
there, the degree of the Picard--Fuchs equation and the dimension of the state
space $H^*(X_W)^{\Aut(W)}$ still match.

       For the quintic three-fold, this argument is
          enough to deduce the LG-CY correspondence for genus-zero Gromov--Witten theory
          entirely. This is not the case in higher dimension.
          The reason is technical: we can only compute the
$J(t,z)$-function  for $t\in H^{1,1}$. In dimension three it turns out
          that all invariants can be deduced from this data.
          In higher dimension, this may well be not enough to
          deduce the entire $J$-function.

          The problem is
          how to recover this information from
          the $B$ model. 
          In collaboration with Iritani \cite{CIR} we partly
          overcome this problem by means of 
          an equivariant theory argument.
This requires technical conditions such as assuming that the ambient space is Gorenstein.
         In both cases, more work is needed in this direction.
          \end{rem}

\subsection{Fulfilling global mirror symmetry in all genera: elliptic orbifold $\PP^1$}
\label{subsect:CYP1}
As we mentioned previously, the physical LG-CY-correspondence requires both $W$ and $G$ to be Calabi--Yau types. While the conjecture is
certainly false without $W$ being Calabi--Yau type,
our theorem on the state space suggests that it may still be true
when $G$ fails to be of Calabi--Yau type.
In this section, we pursue this direction
for $G=\Aut(W)$. The reason for this choice
is simple. By Krawitz's mirror symmetry theorem of type ${\rm LG}{{\mid}}\reflectbox{\rm LG}$,
we have $\Aut(W)^{\vee}=\{1\}$. In this case,
we have a well defined $B$ model for all genera.
Once Problem \ref{prob:orbifolding} is solved,
the same line of research
should hold more generally. Since
$[X_W/\wt{G}]$
is no longer Calabi--Yau for $G=\Aut(W)$, we cannot expect
to have a mirror object of the form of a CY variety. On the other hand the mirror still makes sense.
Indeed, we can exploit the Landau--Ginzburg model.
Consider the LG side illustrated in Section \ref{sect:classMS}.
The state space of $\RW$-theory of $(W, \Aut(W))$
is mirror to the state space of $(W^{\vee}, \{1\})$.
This is a unique situation in which  we have
a $B$ model theory for {\em all genera}.
Here, the
genus-zero theory is Saito's Frobenius manifold
structure on the tangent bundle of the miniversal
deformation of $W^{\vee}$. The higher genus theory
can be obtained via Givental's formalism.

The miniversal deformation is generated by
monomials of the local algebra
$\Qcal_{W^{\vee}}$. Among them, there are
the marginal deformation
generators of the $j_{W^{\vee}}$-invariant
and degree $1$. We label them by $M_{-1},
\dots, M_{-l}$ with $M_{-1}=\prod_i x_i$
and the rest by $M_1, \dots, M_{\mu-l}$
for Milnor number $\mu$. The miniversal deformation is
$$W^{\vee}_{a_{-l}, \dots, a_{-1}, a_0, a_1, \dots,
a_{\mu-l}}=\sum_{-l}^{\mu-l}a_i M_i,$$
where $M_0=W^\vee$.
Usually in singularity theory one considers only germ
of this functions; \emph{i.e.}, one imposes $\abs{a_i}<\epsilon$.
Here, we study global singularity theory to
allow the marginal deformation parameters $a_{<0}$
to vary to infinity. We still require
$\abs{a_i}<\epsilon$ for $i\ge 0$.
The most interesting aspect of this
case is the existence of a rigorous higher genus
theory due to Givental. Recall that the Frobenius
manifold in this case is generically
semisimple (this happens because
we can deform any singularity to Morse singularities).

We should mention that the $B$ model Calabi--Yau three-fold has a higher genus
potential in physics. It is supposed to have many
interesting properties and has been investigated
intensively in recent years by
Klemm and his collaborators \cite{ABK, HKQ}. For
example, it is expected to be nonholomorphic and
satisfy the so called holomorphic anomaly equation.
Furthermore, this antiholomorphic higher genus
generating function is expected to be a modular
form. Unfortunately, we cannot access
these information due to the lack of a mathematically
rigorous definition. On the other hand,
for $[X_W/\wt{G^\vee}]$, we do have a rigorous theory of
on the $B$ side for all genera. This gives us a quite unique
opportunity to study higher genus mirror symmetry. The above
idea has been put into practice by
Krawitz--Shen and Milanov--Ruan in dimension one.

Before describing their work, let us state two mirror
conjectures governing the LG-CY correspondence in
this case. Again, we refer to
$W^{\vee}_{\infty}=\prod_i x_i$ as the large complex
structure point and $W^{\vee}_0=W^{\vee}$ as the Gepner point (see \ref{rem:points}).

\begin{conj}[mirror symmetry ${\rm CY}{{\mid}}\reflectbox{\rm LG}$] \label{conj:MSCYtoLG}
There is a mirror map  matching a neighborhood of
$H^{1,1}_{\CR}([X_W/\wt{G}_W]; \CC)$ with a neighborhood
of the large complex structure point
$W^{\vee}_{\infty}$ such that the genus-$g$ potential
$\Fcal^g_{ GW}$  is matched to the genus-$g$
formal potential $\Fcal^g_{\rm formal}$ of Saito--Givental.
\end{conj}

\begin{conj}[mirror symmetry ${\rm LG}{{\mid}}\reflectbox{\rm LG}$]\label{conj:MSLGtoLG}
Let $G=\Aut(W)$. There is a mirror map matching a neighborhood
of $\Hcal^{1,1}_{W, G}$ with a neighborhood
of the Gepner point
$W^{\vee}_{0}$ such that the $\Fcal^g_{\RW}$-function
is matched to the formal potential $\Fcal^g_{\rm formal}$ of Saito--Givental.
\end{conj}

\begin{conj}\label{conj:Bmodel}
Saito--Givental theory at $W^{\vee}_0$ and $W^{\vee}_{\infty}$
are related by analytic continuation and symplectic transformation.
\end{conj}

\begin{rem}
Naively, one could expect that Saito--Givental theory at $W^{\vee}_0$ and
$W^{\vee}_{\infty}$ are related by analytic
continuation only; however, notice that the precise statement 
involves a subtle issue. Indeed,
the construction of the Frobenius manifold structure
in Saito's theory depends on a
choice of primitive form defined by choosing a
 basis of
middle dimension cycles or period integrals.
Here, we see the similarity between this case and the Calabi--Yau case.
\end{rem}

Now, let us describe the one-dimensional cases. Here, we are
concerned with $[X_W/\wt{G}_W]$ for an elliptic curve $X_W$. We
obtain three examples which share a common feature.
They are one-dimensional stacks of Deligne--Mumford type whose
coarse space is
$\PP^1$ and whose stabilizers are trivial apart from
three special points, whose orbifold structure has orders
$(k_1,k_2,k_3)=(3,3,3), (2,4,4),$ and $(2,3,6)$.
We refer to these orbifolds as Calabi--Yau orbifold $\PP^1$,
where the terminology ``Calabi--Yau'' is justified by the fact that,
although
$\omega$ is not trivial, it becomes trivial after
taking a suitable
tensor power $\omega^{\otimes r}$.

The three cases arise precisely for $W$ equal to
the following polynomials
\begin{align*}
P^{\vee}_8=x_1^3+x^3_2+x^3_3  &&(\emph{i.e.} \ \ (k_1,k_2,k_3)=(3,3,3)\ ),\\
X^{\vee}_9=  x^2_1+x_1x^2_2+x^4_3 &&(\emph{i.e.} \ \ (k_1,k_2,k_3)=(2,4,4)\ ),\\
J^{\vee}_{10}=x^2_1x_2+x^3_2+x^3_3 &&(\emph{i.e.}\  \ (k_1,k_2,k_3)=(2,3,6)\ ).
\end{align*}
Clearly, one should not confuse these orbifolds with
weighted projective stacks. To this effect we adopt the notation
$\PP^1[1/{k_1},1/{k_2},1/{k_3}]$ (instead
of $\PP(w_1,\dots,w_N)$).
The corresponding LG mirrors are the famous simple elliptic singularities
$P_8,X_9,$ and $J_{10}$.

\begin{rem}
It is a natural question to classify all CY orbifold $\PP^1$, \emph{i.e.} 
one-dimensional orbifolds with nontrivial stabilizers only over a
finite number of points and whose canonical line bundle satisfies $\omega^{\otimes r}\cong \Ocal$
for some integer $r$. It turns out that
there are only four possibilities. The
above three cases, alongside with  $\PP^1[\frac12,\frac12,\frac12,\frac12]$.
This stack cannot be expressed as $[X_W/\wt{G}]$ for $G=\Aut(W)$,
but rather as the quotient
of an index two subgroup of $\Aut(W)$.
Again, this requires solving
Problem \ref{prob:orbifolding}.
\end{rem}

As we mentioned previously, the key observation is that the
primitive form and Frobenius structure are determined by a choice of
symplectic basis $\alpha, \beta$ of $H_1$ of the corresponding
elliptic curve. The parameter $a$ together with a symplectic
basis determines
a point $\tau\in \mathbb H_+$ in the upper half-plane. The space
of parameter $a$ can be viewed as the quotient of $\mathbb H_+$
by the monodromy group $\Gamma$.

Saito's Frobenius manifold structure defines the genus-zero
potential function $\Fcal_0(\tau)$. In this situation, Givental
has defined a higher genus generating function $\Fcal_g(\tau)$.
By  studying the transformation of $\Fcal_g$ under
$\tau\mapsto g\tau$ for $g\in \Gamma$, the second author obtains, in collaboration with Milanov, the following theorem.
\begin{thm}\label{thm:Bmodel}
For the miniversal deformation of simple elliptic singularities,
the Saito--Givental function $\Fcal_g$ transforms as  a quasimodular
form of a finite index subgroup of $SL_2( \ZZ)$.
\end{thm}
Recall that the second part of the LG-CY correspondence conjecture
involves the quantization of a
symplectic transformation. The above theorem involves quantization, even if
it is not obvious
from the statement.
Moreover, it provides a much stronger statement; namely,
the claim that the symplectic transformation is related to the modular transformation.

The second theorem fulfilling mirror symmetry
is the following.
\begin{thm}[Krawitz--Shen \cite{KSh}]\label{thm:MSAut}
Both the ${\rm CY}{{\mid}}\reflectbox{\rm LG}$  and the ${\rm LG}{{\mid}}\reflectbox{\rm LG}$ mirror symmetry conjectures hold for all
genera for simple elliptic singularities.
\end{thm}
Theorems \ref{thm:Bmodel} and \ref{thm:MSAut} imply the following corollaries.
\begin{cor}\label{cor:CYorbP1}
The LG-CY correspondence holds for all genera for the CY
orbifold $\PP^1$ of weights $(3,3,3), (2,4,4), (2,3,6)$.
\end{cor}
\begin{cor}\label{cor:modularity}
The generating functions of GW theory for CY orbifold
$\PP^1$ of weights $(3,3,3), (2,4,4), (2,3,6)$ are quasimodular forms
for finite index subgroups of $SL_2( \ZZ)$.
\end{cor}

\begin{rem}
The first corollary provides the first example of the LG-CY correspondence for
all genera.
\end{rem}

The most interesting application is probably the modularity
of the GW theory of orbifold $\PP^1$. As we mentioned in the introduction,
a major problem in geometry and physics is to compute Gromov--Witten
theory. To do so, we often assemble the numerical
Gromov--Witten invariants into a generating function $\Fcal_g$, where
$g$ represents the genus. In some extremely fortunate situations,
$\Fcal_g$ or, more precisely,
the total descendant
potential $\Dcal=\sum_{g\geq 0}h^{g-1}\Fcal_g$
provides a solution to classical integrable systems.
This is the case when the target is a point, by Kontsevich--Witten,
or weighted $\PP^1$ by (work of Okounkov--Pandharipande, Milanov--Tseng, and Johnson).
It is also striking that Okounkov-Pandharipande showed that
$\Fcal_g$ for the elliptic curve $E$ is a quasimodular form of
$SL_2(\ZZ)$. In this way the study of the LG-CY correspondence,
yields another  class of examples: the CY
orbifold $\PP^1$ of weights $(3,3,3), (2,4,4), (2,3,6)$.
In many ways, this is
 much harder to prove because the Chen--Ruan orbifold cohomology
of these examples has more generators
than that of elliptic curves and the Gromov-Witten
invariants are, by definition, more complicated. It would be extremely
interesting to investigate this phenomenon in higher dimension.

\subsection{LG-CY correspondence shortcircuiting mirror symmetry} \label{subsect:CIR}
The genus-zero LG-CY correspondence has not been proven as a tool
to further understand the genus-zero GW theory. The latter was completely elucidated
by  Givental and Lian--Liu--Yau and does not need to be computed.
Instead, we want to use genus-zero information to determine
the symplectic transformation ${\mathbb U}_{\text{\rm LG-CY}}$. Via quantization, ${\mathbb U}_{\text{\rm LG-CY}}$
is expected  to compute $\GW$ theory in higher genus once
we know the higher genus $\RW$ theory. For this
purpose, we need to carry out analytic
continuation via the Mellin--Barnes
method: this yields the desired symplectomorphism as illustrated in Remark \ref{rem:strategy}.
There is an alternative method that allows us to write down the
symplectomorphism $\mathbb U_{\text{\rm LG-CY}}$ directly without passing through the
$B$ model of mirror symmetry and analytic continuation.
As a byproduct, this operation provides an explanation of
the fact that $\UU_{\text{\rm LG-CY}}$ is symplectic.

\subsubsection{Witten's GIT geometric setup}
The construction follows the
purely mathematical description of the LG-CY correspondence, given by Witten in \cite{Wi93b}.
As we recalled (in the homogeneous case) in the introduction there are two GIT quotients for
$\CC^*$ operating on $\CC^N\times \CC$ with weights $(w_1,\dots,w_N,-d)$.
\begin{enumerate}
\item  One of them is the quotient of the open subscheme $(\CC^N\setminus\{\pmb 0\})\times \CC$
yielding the total space of $\Ocal(-N)$ over the weighted projective stack $\PP(\pmb w)$.
This is often referred in the literature as a CY construction
as soon as the sum of the weights is zero. Indeed, consider the complex function
$\ol W=p\textsum_{j=1}^N x^{d/w_j}$ defined in coordinates
$x_1,\dots,x_N,$ and $p$ over $\CC^N\times \CC$ (for simplicity, we are assuming that $d$ is a multiple of $w_j$).
Then, $\ol W$ is $\CC^*$-invariant and descends to the quotient $\Ocal(-N)$.
There, if we consider the map $\ol W$ as a fibration on
$\CC$ we notice that only the special fibre is singular, precisely along the CY hypersurface $X_W$.
\item The second GIT quotient is the quotient of the open subscheme
$\CC^N\times \CC^*$ yielding the stack $[\CC^N/\langle j_W\rangle]$. If we consider
the above map $\ol W$ as a fibration over $\CC$,
we get the LG singularity model $W\colon [\CC^N/\langle j_W\rangle]\to \CC$
for $W=\sum_{j=1}^N x_j^{d/w_j}$
with its isolated singularity in the special fibre and no critical points elsewhere.
\end{enumerate}
From this perspective both sides of the correspondence arise from the same
$\CC^*$-invariant morphism $\ol W$ and from the same geometric setup
\begin{equation}\label{eq:hat}
\ol W\colon [U/\CC^*]\longrightarrow \CC \qquad \qquad (\text{for } U=\CC^N\times \CC).                                                                                         \end{equation}

\subsubsection{Matrix factorizations and Orlov's equivalence}
As illustrated in \cite{HHP} we can exploit the above geometry to present 
the equivalence between the bounded derived
category $\Dcal^b(X_W)$ of coherent sheaves on $X_W$ and
the triangulated category of
graded matrix factorization $MF^{\rm gr}(W)$ of
$W\colon \CC^N\to \CC$ (this follows from Orlov theorem \cite{Orlov} we 
refer to  Isik 
\cite{Isik}  for a complete treatment).
We recall that a \emph{matrix factorization} of $W$ is a pair
$$(E,\delta_E)=\Big(E^0 \overset{\xleftarrow{\ \delta_1\ }}{\underset{\xrightarrow[\ \delta_0\ ]{}}{  }}E^1\Big),$$
where $E=E^0\oplus E^1$ is a $\ZZ_2$-graded finitely generated free module over $R=\CC[x_1,\dots,x_N]$, and
$\delta_E\in {\rm End}^1_R(E)$ is a degree $1\in \ZZ_2$ endomorphism of $E$,
  such that $\delta^2=W\cdot \id_E$.
There is a natural $\ZZ$-graded version, which gives rise to the triangulated category
$MF^{\rm gr}(W)$ of matrix factorizations.

In \cite{PV_MF},
Polishchuk and Vaintrob have shown how to apply the Chern character formalism for
differential graded categories in general to the special case of $\ZZ_d$-equivariant
matrix factorization. Via this construction, and the natural functor
mapping $\ZZ$-graded matrix factorization to $\ZZ_d$-equivariant ones, we get
the Chern character
$$\ch\colon K(MF^{\ZZ_d}(W))\to HH(MF^{\ZZ_d}(W)),$$
where $HH$ stands for the Hochschild cohomology applied to
the differential graded
category of $\ZZ_d$-equivariant matrix factorizations.
In fact, in \cite{PV_MF} a natural isomorphism involving the $\RW$ state space
$$HH(MF^{\ZZ_d}(W))\cong \Hcal_{W,\langle j_W\rangle}$$
is shown. In this way,
Orlov's equivalence
$$MF^{\rm gr}(W)\xrightarrow{\ \ \sim\ \ } \Dcal^b(X_W)$$
yields, after passage to $K$ theory and via Serre duality,
an isomorphism between the state spaces of $\GW$ theory
of $X_W$ and the state space of $\RW$ theory of $W,\Jnw.$
The cohomological version 
of Orlov's equivalence preserves the 
Euler pairings $\chi(E,F) :=\sum_{i\in \ZZ} \dim \Hom(E,F[i])$ after multiplication on both sides by the 
Gamma class $\wt \Gamma$. We refer to \cite{Ir} and \cite{CIR} for precise 
definitions of the Gamma class
 for the  LG model and for the CY hypersurface $X_W$; 
the compatibility with the Euler pairings 
is guaranteed by the following relation with the Todd character: 
$((-1)^{\frac{\deg}2}\wt \Gamma_{X_W})\cdot\wt \Gamma_{X_W}=(2\pi i)^{\deg} \td_{X_W}$
(a consequence of $\Gamma(1 - z)\Gamma(1 + z) = \pi z/ \sin(\pi z)$).
We finally obtain 
\begin{equation}\label{eq:Orlovstates}
\Phi_{\rm Orlov}\colon \Hcal_{W,\Jnw}\longrightarrow H_{\CR}^*(X_W)
\end{equation}
respecting the Euler pairings on both sides.
We point out that this isomorphism does not respect the bigrading
defined in Section \ref{sect:classMS}.

\subsubsection{Short-circuiting mirror symmetry}
For simplicity, and in order to connect to the discussion of the introduction,
let us focus on the case of the quintic three-fold and refer to
\cite{CIR} for the generalizations to weighted homogeneous polynomials.

Recall that solving $\GW$ and $\RW$ theory
amounts to writing a basis of flat sections of
a certain vector bundle with connection. Namely,
once the state space of the theory is specified $H_{\W}$,
we consider the trivial vector bundle $D_{\W}=H_{\W}\times \Acal\longrightarrow \Acal$,
where $\Acal$ is a contractible neighbourhood
of $H_{\W}^{1,1}$.
This vector bundle is equipped with Dubrovin's connection $\nabla_{\W}$ and its fibres are the even-degree parts of
the state spaces of the relevant theory:
 $H_{\CR}^*(X_W)$ for $\GW$ theory and $\Hcal_{W,\Jnw}$ for $\RW$ theory.
Solving each theory in genus zero amounts to define morphisms
\begin{equation*}
\xymatrix@R=.8cm{
\Gamma(D_{\RW},\nabla_{\RW})&& \Gamma(D_{\GW},\nabla_{\GW})\\
H_{\RW}\ar[u]&& H_{\GW}\ar[u]
}
\end{equation*}
identifying $H_{\W}$ with the space of flat sections $\Gamma(D_{\W},\nabla_{\W})$.
The proofs of the
mirror symmetry conjectures \ref{conj:MSCYtoCY} (Givental \cite{Givental},
Lian--Liu--Yau \cite{LLY}) and \ref{conj:MSCYtoLG} (by the authors \cite{ChRu})
yield an identification of $(D_{\W},\nabla_{\W})$ with two local pictures
of the $B$ model vector bundle $(V,\nabla_V)$. Via analytic continuation this yields
an identification between $\Gamma(D_{\RW},\nabla_{\RW})$ and
$\Gamma(D_{\GW},\nabla_{\GW})$. More precisely, following Figure \ref{fig:picture},
we can identify both spaces of flat sections to germs of flat sections of $(V,\nabla_V)$
around $0$ and $\infty$ and carry out a parallel transport there (this is well defined
in terms of multivalued functions, or in terms of a single-valued function
once branch cut is chosen, see Remark \ref{rem:monod}).
In \cite{CIR} we prove, in collaboration with Iritani, that analytic continuation can be
equivalently replaced by Orlov's isomorphism \eqref{eq:Orlovstates}.
\begin{thm}
The diagram
\begin{equation*}
\xymatrix@R=.2cm{
\Gamma(D_{\RW},\nabla_{\RW})\ar[rrrr]^{\rm analytic\ continuation}&&&& \Gamma(D_{\GW},\nabla_{\GW})\\
&& \square&& \\
H_{\RW}\ar[uu]\ar[rrrr]_{\Phi_{\rm Orlov}}&&&& H_{\GW}\ar[uu]
}
\end{equation*}
is commutative. In this way the linear transformation  $\UU_{\text{\rm LG-CY}}$
matching the bases of
flat sections is encoded by a symplectic matrix
expressing $\Phi_{\rm Orlov}$ for a given choice of bases of the two state spaces.
\end{thm}
\begin{rem}\label{rem:HHP}
The above statement may be regarded as saying that Orlov's categorical
equivalence \emph{mirrors} on the $A$-model
the analytic continuation carried out
by means of the $B$-model picture $(V,\nabla_V)$. This
fits in Iritani's framework developped in \cite{Ir}
describing the integral structures mirroring the
local systems $H^{3}(X_{W,t}^{\vee} , \ZZ) \subset H^{3}(X_{W,t}^{\vee} , \CC)$ of the $B$ models.

In physics, this counterpart to parallel transport has been widely treated.
Hori, Herbst, and Page rephrase Orlov's equivalence in terms of
\emph{brane transport}, see \cite{HHP}.
One of the most
interesting aspects of their work is the above mentioned reformulation of Orlov's functor
passing through the geometric setup \eqref{eq:hat}.
There, we are led to extend
representations of the cyclic group $\Jnw$ to representations of $\CC^*$; clearly,
there is not a unique way to do so and this is the reason why
Orlov's equivalence should be actually regarded as a set of functors
$$MF^{\rm gr}(W)\xrightarrow{\ \sim \ } \Dcal^b(X_W)\qquad
\text{yielding} \qquad \Phi_{a}\colon H_{\RW}\longrightarrow H_{\GW}$$
parametrized by $a\in \ZZ$ (see \cite[\S2.2]{Orlov}).
Any two of these functors match for
a suitable autoequivalence of
the source category. We treat this aspect completely in \cite{CIR}.
\end{rem}

\begin{rem}\label{rem:monod}
In complete analogy, the analytic continuation should be carried along an open substack
of the one dimensional stack $[\PP^1/\ZZ_d]$; this happens because over the conifold point and over the
large complex structure point, the fibres $X_{W,\pmb a}^\vee$ are singular and over the
Gepner point there is a nontrivial stabilizer.
In this way, the analytic continuation of period integrals is defined
up to the monodromy operator $T$ at infinity (Remark \ref{rem:monod}).
The theorem above should be more precisely stated as follows:
there is a choice of analytic continuation commuting with Orlov's isomorphism  $\Phi_0$.
Let us express Orlov's isomorphisms as symplectic matrices $\UU_a$
with respect to the chosen bases for $H_{\RW}$ and $H_{\GW}$.
The linear map $T$ operates on the period integrals and changes the symplectomorphism
$\UU_{\text{\rm LG-CY}}$ by conjugation. Then, we have the identification
$$\Phi_a=T^{-a}\UU_{\text{\rm LG-CY}}T^a$$
via the morphism
$H_{\GW}\to \Gamma(D_{\GW},\nabla_{\GW})$.
\end{rem}

\small{

}

\vspace{.5cm}

\noindent \textsc{Institut Fourier, UMR du CNRS 5582,
Universit\'e de Grenoble 1,
BP 74, 38402,
Saint Martin d'H\`eres,
France}\\
\textit{E-mail address:} \url{chiodo@ujf-grenoble.fr}

\vspace{.3cm}

\noindent \textsc{Department of Mathematics, University of Michigan, Ann Arbor, MI 48109-1109,
USA} and \\ \textsc{Yangtze Center of Mathematics, Sichuan University, Chengdu,
610064, P.R. China}\\
\textit{E-mail address:} \url{ruan@umich.edu}

\begin{thebibliography}
{AAAAA}
\bibitem[AGrV08]{AGrV} \textsc{D.~Abramovich, T.~Graber, A.~Vistoli,}
\emph{Gromov--Witten theory of Deligne--Mumford stacks}, Amer.
J.~Math.~, {\bf130}, Number 5, (2008).

\bibitem[AGuV88]{AGV}
\textsc{V.I.~Arnold, S.M.~Gusein-Zade, A.N.~Varchenko,}
\emph{Singularities of differentiable
maps}, volume II, Birkh\"auser, Boston 1988.

\bibitem[AJ03]{AJ}
\textsc{D.~Abramovich, T.~J.~Jarvis}, \emph{Moduli of twisted spin
curves}, Proc. Amer. Math. Soc. {\bf131} (2003), 685--699,
Preprint version: \url{math.AG/0104154}.

%






\bibitem[ABK08]{ABK}\textsc{M.~Aganagic, V.~Bouchard, A.~Klemm,}
\emph{Topological Strings and (Almost) Modular Forms,}
Commun.~Math.~Phys.~{\bf277}, 771--819, (2008). Preprint version:
\url{hep-th/0607100}.


\bibitem[BB97]{BB} \textsc{V.~V.~Batyrev, L.~A.~Borisov},
\emph{Dual Cones and Mirror Symmetry for Generalized Calabi--Yau
Manifolds}, Mirror Symmetry II, AMS/IP Stud.~Adv.~Math 1, Amer.
Math. Soc. Providence, RI (1997), 71--86.

\bibitem[BCOV94]{BCOV} \textsc{M. Bershadsky, S. Cecotti, H. Ooguri, C. Vafa},
\emph{Kodaira--Spencer Theory
of Gravity and Exact Results for Quantum String Amplitudes},
Comm. Math. Phy. 165(1994)311-427

\bibitem[BH93]{BH} \textsc{P.~Berglund, T.~H\"ubsch},
\emph{A Generalized Construction of Mirror Manifolds}, Nuclear
Physics B, vol {\bf393}(1993), 397--391.

\bibitem[BK97]{BK} \textsc{P.~Berglund, S.~Katz},
\emph{Mirror Symmetry Constructions: A Review},
Mirror Symmetry II, AMS/IP Stud.~Adv.~Math 1, Amer. Math.
Soc. Providence, RI (1997), 71--86,
Preprint version: \url{arXiv:hep-th/9406008}.

\bibitem[BMP09]{BMP}
\textsc{S.~Boissi\`ere, \'E.~Mann and F.~Perroni}, \emph{$A$ model for
the orbifold Chow ring of weighted projective spaces},
Communications in Algebra, {\bf37} (2009), 503--514.

\bibitem[Bo]{Bo} \textsc{L.~Borisov},
\emph{Berglund--H\"ubsch mirror symmetry via vertex algebras},
Preprint version: \url{arXiv:1007.2633v3}

\bibitem[CDGP91]{CDGP}
\textsc{P.~Candelas, X.~C.~De La Ossa, P.~S.~Green, L.~Parkes},
\emph{A pair of Calabi--Yau manifolds as an exactly soluble
superconformal theory,} Nucl. Phys. B {\bf359} (1991) 21--74.


\bibitem[Ch06]{chK} \textsc{A.~Chiodo,}
\emph{The Witten top Chern class via $K$-theory.} J. Algebraic
Geom. {\bf 15} (2006), no.~4, 681--707. Preprint version:
\url{math.AG/0210398}.

\bibitem[Ch08a]{chstab} \textsc{A.~Chiodo,}
\emph{Stable twisted curves and their $r$-spin structures (Courbes
champ\^etres stables et leurs structures $r$-spin).} Ann.~Inst.~Fourier, Vol. {\bf58}
no. 5 (2008), p. 1635--1689. Preprint version:
\url{math.AG/0603687}.

\bibitem[Ch08b]{ch} \textsc{A.~Chiodo,}
\emph{Towards an enumerative geometry of the moduli space of
twisted curves and $r$th roots.} Compos. Math. {\bf 144} (2008),
Part 6, 1461--1496. Preprint version: \url{math.AG/0607324}.

\bibitem[CIR]{CIR} \textsc{A.~Chiodo, H.~Iritani, Y.~Ruan},
\emph{Landau-Ginzburg/Calabi-Yau correspondence, global mirror symmetry and Orlov equivalence},
Preprint \url{arXiv:1201.0813}.

\bibitem[ChiR10]{ChRu} \textsc{A.~Chiodo and  Y.~Ruan},
\emph{Landau--Ginzburg/Calabi--Yau correspondence for quintic
three-folds via symplectic transformations}, Invent. Math.
 (2010) {\bf182}, 117--165.
Preprint version: \url{arXiv:0812.4660}.


\bibitem[ChiR11]{CR_AIM}
\textsc{A.~Chiodo, Y.~Ruan},
\emph{LG/CY correspondence: the state space isomorphism},
{Adv. Math.}, {227}, Issue 6 (2011), 2157-2188

\bibitem[Cl]{CL}\textsc{P. ~Clarke},
\emph{Duality for toric Landau-Ginzburg models},
Preprint: \url{arXiv:0803.0447}.


\bibitem[CZ10]{CZ} \textsc{A.~Chiodo, D.~Zvonkine},
\emph{Twisted Gromov--Witten $r$-spin potentials and Givental's
quantization}.
Adv.~Theor.~Math.~Phys. Volume {\bf13}, Number 5 (2009), 1335--1369.
Preprint version: \url{arXiv:0711.0339}.

\bibitem[Co09]{Coates} \textsc{T.~Coates},
\emph{On the Crepant Resolution Conjecture in the Local Case},
Communications in Mathematical Physics, {\bf287}(2009) 1071-1108.

\bibitem[CCLT09]{CCLT} \textsc{T.~Coates, A.~Corti, Y.-P.~Lee, H.-H. Tseng},
\emph{The quantum orbifold cohomology of weighted projective spaces}.
Acta Math. {\bf202} (2009), no. 2, 139--193

\bibitem[CIT09]{CIT} \textsc{T.~Coates, H.~Iritani, H.-H.~Tseng},
\emph{Wall-Crossings in Toric Gromov–Witten Theory I: Crepant Examples}.
Geometry and Topology, {\bf13} (2009), 2675--2744.


\bibitem[CCIT09]{CCIT} \textsc{T.~Coates, A.~Corti, H.~Iritani, H.-H.~Tseng},
\emph{Computing Genus-Zero Twisted Gromov--Witten Invariants}, Duke Math.
{\bf147}(2009), 377--438.

\bibitem[CG07]{CG}
\textsc{T.~Coates, A.~Givental,} \emph{Quantum Riemann--Roch,
Lefschetz and Serre}, Annals of mathematics, vol.
{\bf 165}, no 1, 2007
15--53.

\bibitem[CR]{CR} \textsc{T.~Coates, Y,~Ruan},
\emph{Quantum Cohomology and Crepant Resolutions: A
Conjecture}, Preprint: \url{arXiv:0710.5901}.

\bibitem[De97]{De}
\textsc{P.~Deligne},
\emph{Local behavior of Hodge structures at infinity.}
in Mirror Symmetry II, pp. 683--699,
ed. B. Greene and S. T. Yau, AMS and International Press, 1997.


\bibitem[Do82]{Do}
\textsc{I.~Dolgachev},
\emph{Weighted projective varieties},
Proc. Vancouver 1981, Lecture Notes in Math.,
Vol. {\bf956}, Springer, 1982, pp. 34--71.


\bibitem[FSZ]{FSZ} \textsc{C.~Faber, S.~Shadrin, D.~Zvonkine,}
\emph{Tautological relations and the r-spin Witten conjecture.}
Annales Scientifiques de l'ENS {\bf43}, fascicule 4 (2010), 621--658.
Preprint version: \url{math.AG/0612510}.

\bibitem[FJR2]{FJR2}\textsc{H.~Fan, T.~Jarvis, Y.~Ruan,}
\emph{Geometry and analysis of spin equations}. Comm. Pure Appl.
Math. 61 (2008), no. 6, 745--788.

\bibitem[FJR1]{FJR1}\textsc{H.~Fan, T.~Jarvis, Y.~Ruan,}
\emph{The  Witten equation, mirror symmetry and quantum
singularity theory.} Preprint: \url{arXiv:0712.4021v1}.

\bibitem[FJR3]{FJR3} \textsc{H.~Fan, T.~Jarvis, Y.~Ruan,}
\emph{The Witten equation and its virtual fundamental cycle},
Preprint: \url{arXiv:0712.4025}.

\bibitem[FJMR]{JRbroadD4}
\textsc{H.~Fan, T.~J.~Jarvis, E.~Merrell, Y.~Ruan,} 
\emph{Witten's $D_4$ Integrable Hierarchies Conjecture}
Preprint: \url{arXiv:1008.0927}.


\bibitem[Gi96]{Gi}\textsc{A.~Givental,}
\emph{A mirror theorem for toric complete intersections.}
Topological field theory, primitive forms and related topics
(Kyoto, 1996), 141--175, Progr. Math., 160.

\bibitem[Gi04]{Givental} \textsc{A.~Givental,}
\emph{Gromov--Witten invariants and quantization of quadratic
hamiltonians.} In ``Frobenius manifolds'',  91--112, Aspects
Math., E36, Vieweg, Wiesbaden, 2004, Preprint version:
\url{math.AG/0108100}.





\bibitem[GMP95]{GMP} \textsc{B.~R.~Greene, D.~R.~Morrison; M.~R.~Plesser,}.
 \emph{Mirror manifolds in higher dimension}. Comm. Math. Phys. 173 (1995),
no. 3, 559--597

\bibitem[He03]{Hertling} \textsc{C.~Hertling,} \emph{$tt^∗$ geometry,
Frobenius manifolds, their connections and their construction for
singularities}.
J. Reine Angew. Math. {\bf555}, 2003, 77--161.

\bibitem[HHP]{HHP} \textsc{M.~Herbst, K.~Hori, D.~Page},
\emph{Phases Of N=2 Theories In 1+1 Dimensions With Boundary},
DESY-07-154, CERN-PH-TH/2008-048
Preprint version: \url{arXiv:0803.2045}

\bibitem[HW04]{H} \textsc{K.~Hori, J.~Walcher}, \emph{D-branes from matrix factorizations.}
Strings 04. Part I. C. R. Phys. 5 (2004), no. 9-10, 1061--1070.



\bibitem[Ho]{Ho} \textsc{P.~Horja},
\emph{Hypergeometric functions and mirror symmetry in toric
varieties}, Preprint: arXiv:math/9912109.

\bibitem[HKQ]{HKQ} \textsc{M.~Huang, A.~Klemm, S.~Quackenbush},
\emph{Topological string theory on compact Calabi--Yau:
modularity and boundary conditions,} Homological mirror symmetry, 45--102, Lecture Notes in Phys., 757, Springer, Berlin, 2009
\url{arXiv:hep-th/0612125}.




\bibitem[JKV01]{JKV} \textsc{T.~J.~Jarvis, T.~Kimura, A.~Vaintrob,}
\emph{Moduli spaces of higher spin curves and integrable hierarchies.}
Compositio Math. {\bf 126}  (2001),  no.~2, 157--212,
\url{math.AG/9905034}.

\bibitem[IV90]{IV} \textsc{K.~Intriligator, C.~Vafa}, \emph{Landau--Ginzburg
orbifolds}, Nuclear Phys. B {\bf339} (1990), no 1, 95--120.

\bibitem[Ir09]{Ir}\textsc{I.~Iritani},
\emph{An integral structure in
quantum cohomology and mirror symmetry for orbifolds},
Adv. in Math. {\bf 222} (2009) 1016--1079
Preprint version: \url{arXiv:0903.1463v1}.

\bibitem[Is]{Isik}\textsc{M.~U.~Isik},
\emph{Equivalence of the derived category of a variety with a singularity category}, 
Preprint: \url{arXiv:1011.1484}.


\bibitem[Ka06]{KA1} \textsc{Kaufmann, Ralph}, \emph{Singularities with symmetries,
orbifold Frobenius algebras and mirror symmetry}. Contemp. Math., {\bf403} (2006), 67-116

\bibitem[Ka]{KA2} \textsc{Kaufmann, R},
\emph{A note on the two approaches to stringy
functors for orbifolds}, Preprint: \url{arXiv:math/0703209}

\bibitem[Ko92]{Ko} \textsc{M.~Kontsevich}, \emph{Intersection theory on
the moduli space of curves and the matrix Airy function.} Comm.
Math. Phys. 147 (1992), no. {\bf1}, 1--23.

\bibitem[Ko]{Ko1} \textsc{M.~Kontsevich}, unpublished.

\bibitem[Kr]{Kr} \textsc{M.~Krawitz}, \emph{FJRW rings and Landau--Ginzburg Mirror Symmetry},
Preprint: \url{arXiv:0906.0796}.

\bibitem[KPABR10]{KPABR} \textsc{Marc Krawitz, Nathan Priddis, Pedro Acosta, Natalie Bergin, Himal Rathnakumara},
\emph{FJRW-rings and Mirror Symmetry}, Comm. Math. Phys. 296(2010) 145-174

\bibitem[KS]{KSh} \textsc{M.~Krawitz, Y.~Shen},
\emph{Landau-Ginzburg/Calabi-Yau correspondence of all genera for
elliptic orbifold $\PP^1$}, 
Preprint: \url{arXiv:1106.6270}

\bibitem[KM94]{KM} \textsc{M.~Kontsevich, Y.~Manin},
\emph{Gromov-Witten classes,
quantum cohomology, and enumerative geometry}, Commun.~Math.~Phys.~{\bf164} (1994) 525--562

\bibitem[KS92]{KS} \textsc{M.~Kreuzer and H.~Skarke},
\emph{On the classification of quasihomogeneous functions}, Comm.
Math. Phys. {\bf 150} (1992), no. 1, 137--147.

\bibitem[KS93]{KStest} \textsc{M.~Kreuzer and H.~Skarke},
\emph{All abelian symmetries of Landau-Ginzburg potentials},
Nucl.~Phys.~B \textbf{405} (1993), no 2-3, 305--325. Preprint version:
\url{hep-th/9211047}.

\bibitem[LLY97]{LLY} \textsc{B.~Lian, K.~Liu, S.~Yau}, \emph{Mirror principle. I.}
 Asian J. Math. 1 (1997), no. 4, 729--763

\bibitem[LR01]{LR} \textsc{A.~Li, Y.~Ruan}, \emph{Symplectic
surgeries and Gromov--Witten invariants of Calabi--Yau three-folds},
Invent. Math. 145, 151-218(2001)


\bibitem[Lo84]{Loo} \textsc{E.J.N. Looijenga},
\emph{Isolated singular points on complete intersections},
London Math.Soc. Lecture Note Series {\bf77},
Cambridge University Press 1984.

\bibitem[Mo93]{MO}
\textsc{D.R.~Morrison,}
\emph{Beyond the K\"ahler cone}.
Proceedings of the Hirzebruch 65
Conference on Algebraic Geometry
(Ramat Gan, 1993), 361--376, Israel Math. Conf. Proc., 9

\bibitem[MP06]{MP} \textsc{D.~Maulik, R.~Pandharipande}, \emph{A topological view of
Gromov--Witten theory.} Topology {\bf45} (2006), no. 5, 887--918.

\bibitem[MR]{MR} \textsc{T.~Milanov, Y.~Ruan}, \emph{Gromov--Witten theory
of elliptic orbifold $\PP^1$ and quasi-modular forms}, 
Preprint: \url{arXiv:1106.2321}


\bibitem[Mo97]{Morr}
\textsc{D.~R.~Morrison},
\emph{Mathematical Aspects of Mirror Symmetry}, in
Complex Algebraic Geometry (J. Koll\'ar, ed.),
IAS/Park City Math. Series, vol. {\bf3}, 1997, 265--340

\bibitem[OS78]{OS} \textsc{P.~Orlik, L.~Solomon},
\emph{Singularities II;
Automorphisms of forms.} Math.
Ann. {\bf231} (1978), 229--240.

\bibitem[Or]{Orlov}
\textsc{D.~Orlov},
\emph{Derived categories of coherent sheaves and triangulated
categories of singularities}, Preprint: \url{math.AG/0503632}.


\bibitem[Ph85]{Pham} \textsc{F.~Pham}, \emph{La descente des cols par les onglets de Lefschetz,
avec vues sur Gauss--Manin}, Syst\`emes diff\'erentiels et singularit\'es.
Asterisques {\bf 130} (1985), 11--47.


\bibitem[Po04]{Po} \textsc{A.~Polishchuk},
\emph{Witten's top Chern class on the moduli space of higher spin
  curves.} Frobenius manifolds,  253--264,
Aspects Math., E36, Vieweg, Wiesbaden, 2004.
Preprint version: \url{math.AG/0208112}.

\bibitem[PV01]{PV} \textsc{A.~Polishchuk, A.~Vaintrob},
\emph{Algebraic construction of Witten's top Chern class.}
Advances in algebraic geometry motivated by physics (Lowell, MA,
2000), 229--249, Contemp. Math., {\bf 276},
Amer. Math. Soc., Providence, RI, 2001.
Preprint version: \url{math.AG/0011032}.


\bibitem[PV]{PV_MF}
\textsc{A.~Polishchuk, A.~Vaintrob.}
\emph{Chern .}
Preprint version: \url{math.AG/0011032}.

\bibitem[Ru]{R} \textsc{Y.~Ruan}, \emph{The Witten equation and
geometry of Landau--Ginzburg model}. in preparation.


\bibitem[St77]{St} \textsc{J.~Steenbrink},
\emph{Intersection form for quasi-homogeneous singularities}.
Compositio Mathematica, {\bf34} no. 2 (1977), p. 211--223



\bibitem[VW89]{VW89} \textsc{C.~Vafa and N.~Warner,} \emph{Catastrophes and the
classfication of conformal field theories}, Phys. Lett. 218B
(1989) 51.

\bibitem[Wa80]{Wa1} \textsc{C.~T.~C.~Wall}, \emph{A note on symmetry
of singularities}, Bull.~London Math.~Soc.~\textbf{12} (1980), no. 3, 169--175.



\bibitem[Wi91]{Wi1} \textsc{E.~Witten,} \emph{Two-dimensional gravity and
intersection theory on the moduli space}, Surveys in Diff.~Geom.~\textbf{1} (1991), 243--310.

\bibitem[Wi93a]{Wi2} \textsc{E.~Witten,} \emph{Algebraic geometry associated with
matrix models of two-dimensional gravity}, Topological models in
modern mathematics (Stony Brook, NY, 1991), Publish or Perish,
Houston, TX, 1993, 235--269.

\bibitem[Wi93b]{Wi93b} \textsc{E.~Witten,}  \emph{Phases of $N=2$ theories in two
dimensions}, Nucl.Phys.~B \textbf{403} (1993), 159-222.

\bibitem[Zi08]{Zi} \textsc{A.~Zinger,} \emph{Standard vs.~reduced genus-one Gromov--Witten
invariants}, Geom.~Topol.~ {\bf12}, 2 (2008),
1203--1241.

\end{thebibliography}
\end{document}